\documentclass[preprint,12pt]{elsarticle}
\usepackage{lineno}
\usepackage[breaklinks]{hyperref}
\usepackage{amssymb}
\usepackage{subcaption}
\usepackage{ntheorem}
\usepackage{amsmath}
\usepackage[title]{appendix}
\usepackage{mathrsfs}
\usepackage{multirow,booktabs}
\theorembodyfont{\upshape}
\theoremseparator{.}
\newtheorem{theorem}{Theorem}[section]
\newdefinition{remark}{Remark}[section]
\newtheorem{proposition}{Proposition}[section]
\newtheorem{lemma}{Lemma}[section]
\newtheorem{definition}{Definition}[section]

\newtheorem*{proof}{Proof}
\numberwithin{equation}{section}
\begin{document}

\begin{frontmatter}

\title{Boundedness and exponential stabilization for  time-space
fractional parabolic-elliptic Keller-Segel model in higher dimensions}

\author[mymainaddress]{Fei Gao\corref{mycorrespondingauthor}}
\ead{gaof@whut.edu.cn}
\author[mymainaddress]{Hui Zhan}
\ead{2432593867@qq.com}

\cortext[mycorrespondingauthor]{Corresponding author.}

\address[mymainaddress]{Department of Mathematics and Center for Mathematical Sciences, Wuhan University of Technology, Wuhan, 430070, China}
\begin{abstract}
For the  time-space fractional degenerate Keller-Segel equation
\begin{equation*}
    \begin{cases}
          \partial _{t}^{\beta }u=-(-\Delta )^{\frac{\alpha}{2}}(\rho (v)u),& t>0\\
(-\Delta )^{\frac{\alpha}{2}} v+v=u,&  t>0
    \end{cases}
\end{equation*}
$x\in\Omega, \Omega \subset \mathbb{R}^{n}, \beta\in (0,1),\alpha\in (1,2)$, we consider for $n\geq 3$ the problem of finding a time-independent upper bound of the classical solution such that  as $\theta>0,C>0$
 \begin{equation*}
     \left \| u(\cdot ,t)-\overline{u_{0}} \right \|_{L^{\infty }(\Omega )}+\left \| v(\cdot ,t)-\overline{u_{0}} \right \|_{W^{1,\infty }(\Omega )}\leq Ce^{(-\theta)^{1/\beta}t},
 \end{equation*}
where $\overline{u_{0}}=\frac{1}{\left | \Omega  \right |}\int _{\Omega }u_{0}dx$. We find such solution in the special cases of time-independent upper bound of the concentration  with  Alikakos-Moser iteration and fractional differential inequality. In those cases the problem is reduced to a time-space fractional parabolic-elliptic equation which is treated with Lyapunov functional  methods. A key element in our construction is a proof of the exponential
stabilization toward the constant steady states  by using fractional Duhamel type integral equation.
\end{abstract}


\end{frontmatter}

\section{Introduction}
We consider the class of classcial solutions of the  parabolic-
elliptic time-space fractional Keller-Segel model 
\begin{equation}\label{1.1}
    \begin{cases}
        \partial _{t}^{\beta }u=-(-\Delta )^{\frac{\alpha}{2}}(\rho(v)u),& x\in\Omega ,t>0\\
(-\Delta )^{\frac{\alpha}{2}} v+v=u,& x\in\Omega ,t>0\\
\partial_{v}u=\partial_{v}v=0,&x\in\partial \Omega ,t>0\\
u(x,0)=u_{0}(x),&x\in\Omega
    \end{cases}
\end{equation}
where $\beta\in (0,1),\alpha\in (1,2)$. The precise definition of the fractional Laplacian is given in Appendix \ref{8i90oyt}. For $\beta=1,\alpha=2$ we recover the parabolic-elliptic simplification of the original
Keller-Segel model with signal-dependent motility for local sensing chemotaxis, whose theory is well known in \cite{jiang2022boundedness}. We also assume that we are given initial data
\begin{equation}\label{90oiutgv}
    u_{0}\in C^{0}(\overline{\Omega }),\qquad u_{0}\geq 0,\qquad u_{0}\not\equiv 0.
\end{equation}
 However, the signal-dependent decreasing function $\rho(\cdot )$ has recently attracted some attention in theoretical analysis. For facilitate the need for proof theorem \ref{theorem4}, we have listed in Appendix \ref{ijugggrc} the most relevant sources to the applications that we know of. Compared with reference \cite{jiang2022boundedness}, the paper obtains more meaningful results of the solution of time-space fractional Keller-Segel in the sense of fractional derivative, and uses fractional order related inequalities in the proof process to simplify the problem.

 We refer to \cite{kemppainen2017representation} for the basic theory of existence and uniqueness of classical solutions for equation \eqref{1.1}. There were many literature discussing the existence and uniqueness of solutions to the  time-space fractional Keller-Segel equation \cite{li2018cauchy,jiang2021weak,nguyen2023cauchy}. Li et al \cite{li2018cauchy} prove the unique existence, blow-up behavior and nonegativity preservation of mild solution was obtained in $L^{p_{c}}$ space, where $p_{c}$ is the critical index. Jiang and Wang \cite{jiang2021weak} discussed  the existence,mass conservation and decay properties of weak solutions to  time-space fractional  parabolic-elliptic  Keller-Segel model. Nguyen et al \cite{nguyen2023cauchy} study Cauchy problem of the  time-space  fractionalKeller-Segel model in a critical homogeneous Besov space with the assumption that the initial datum is sufficiently small  by Banach fixed point theorem, some special functions.

Let us first recall some previous results on fractional Keller-Segel equation. Since there is a large amount of papers for these equations, we mention the ones related to our results. The results take different forms according to the value of the exponent $\alpha,\beta$. For higher dimensions $n\geq 3$ and $\alpha=2,\beta=1$, the boundedness was studied in several works provided that $\rho$ satisfies some algebraically decreasing assumptions \cite{ahn2019global,fujie2021boundedness,wang2021parabolic}. When $\alpha=2$ and $\beta\in(0,1)$, the fractional operator $-(-\Delta)^{\frac{\alpha}{2}}$ become the standard Laplacian $\Delta$, which is a time fractional Keller-Segel model. Such equation were obtained for the fractional PDEs by Azevedo et al \cite{azevedo2019existence} and  Cuevas, Claudio et al \cite{cuevas2020time}. When $\beta=1$ and $\alpha\in (1,2)$, the problem \eqref{1.1} reduces to the fractional Keller-Segel equation. Recently, there are some results involved in the study of the fractional Keller–Segel equation (see \cite{zhu2020decay,huang2016well,escudero2006fractional,biler2009two,zhang2019global} for example).

 However, in the sense of fractional derivative $\alpha\in(1,2),\beta\in(0,1)$, there is little research on the proof of the boundedness and exponential stabilization for solutions of such equations \eqref{1.1}. Therefore, the paper will start from the following aspects. First, the use of suitable inequalities about fractional derivative were allowed to prove  time-independent upper bounds of $v$  that enter substantially into the derivation of the main results. Next, we prove the uniform-in-time boundedness results for \eqref{1.1} under an assumption (H1) in Appendix \ref{367897} by Lyapunov functional and Alikakos–Moser type iteration. Finally,  with the aid of the Lyapunov functional again and fractional Duhamel type integral equation, exponential stabilization of the global solution toward the spatially homogeneous steady states
 $(\overline{u_{0}},\overline{u_{0}})$ is obtained for the first time.

 We remind that fractional differential inequalities are a standard tool used to develop further theory, see in that respect the work of  Alsaedi Ahmed for  inequality in fractional calculus \cite{Ahmed2017A,2015Maximum}. Let us quote three examples of application of the line of results of this paper that are already available: the existence and uniqueness of classical solution, which we do in Appendix \ref{90oiuy}; understanding the hypothesis of the function $\rho(\cdot )$ in Appendix \ref{367897}; dealing with Alikakos Moser iteration problems \cite{alikakos1979application}.

\textbf{Outline of the paper and main results.}
The main purpose of the present paper is obtaining  boundedness  and exponential stabilization for the solutions of the problem \eqref{1.1} in the meaning of time-space fractional derivatives. 

In Section \ref{89iuyh}, we aim to establish the uniform-in-time upper bound for $v$ in meaning of fractional derivatives. The proof consists of several steps. First, we would like to recall the following identity, which unveils the key mechanism of our system:
\begin{equation}\label{0oiuyybg}
     _{0}^{C}\textrm{D}_{t}^{\beta }v+\rho  (v)u=(I+(-\Delta )^{\frac{\alpha}{2}} )^{-1}[\rho  (v)u].
\end{equation}
Here $(-\Delta )^{\frac{\alpha}{2}} $ denotes the fractional Laplacian.
Thanks to the comparison principle of fractional elliptic equations and the decreasing property of $\rho$, one has that
$$(I+(-\Delta )^{\frac{\alpha}{2}} )^{-1}[\rho  (v)u]\leq (I+(-\Delta )^{\frac{\alpha}{2}} )^{-1}[\rho  (v_{*})u]=\rho  (v_{*})v$$
with $v^{*}$ being the strictly positive lower bound for $v$ given by Lemma  \ref{lemma2.2} below. And by using fractional differential inequalities, we
obtains  a point wise upper bound of $v$ with generic functions satisfying (H0) in Appendix \ref{367897}. Then, on the one hand, with the help of fractional inequalities in Appendix \ref{inequaliyy}-\ref{i9io87h}, the system \eqref{1.1} possesses a Lyapunov functional such that
\begin{align}\label{90oiuuytrfbn}
 \nonumber&\frac{1}{2}\tilde{\partial }_{T}^{\beta }\left ( \left \| D^{\alpha/2 }v  \right \|^{2}+\left \| v \right \|^{2} \right )+\int _{\Omega }\rho  (v)\left | D^{\alpha }v \right |^{2}dx\\
    &\leq-\int _{\Omega }(\rho (v)+v{\rho  }'(v))\left | D^{\alpha/2 }v \right |^{2}dx.
\end{align}
which implies a time-independent estimate of $\underset{t\geq 0}{\text{sup}}\left \| v \right \|_{H^{1}(\Omega )}$ under the assumption (H1) in Appendix \ref{367897}. On the other hand, under the assumption (H2) in Appendix \ref{367897}, we proceed to derive the uniform-in-time upper bounds of $v$ which based on a delicate Alikako-Moser iteration \cite{alikakos1979application}.

In Section \ref{iokjuyhnnj},in order to prove the boundedness of solutions, we use  the time-independent upper bound of $v$ of Section \ref{89iuyh} and fractional differential inequalities in Appendix \ref{inequaliyy}-\ref{i9io87h}. Next, With the aid of the second equation of \ref{1.1} and equation \eqref{0oiuyybg}, we construct an estimation involving a
weighted energy $\int _{\Omega }u^{p}\rho ^{q}(v)dx$ for some $1+p>\frac{n}{2}$ and $q>0$ by using the decreasing property of $\rho(\cdot)$ that is bounded. Thus, we can to proceed to deduce the uniform-in-time boundedness of the solution.

Section \ref{90oiuytgv}  studies the exponential stabilization of the global solution relying on Lyapunov functional \eqref{3.1} and fractional Duhamel type integral equation. First, since $\overline{v(t)}=\overline{u_{0}}$, we have $_{0}^{C}\textrm{D}_{t}^{\beta }\int _{\Omega }v(t)dx=0$ for all $t>0$, and Thus, we can infer from \eqref{90oiuuytrfbn} that
\begin{align*}
   &\frac{1}{2}\,_{0}^{C}\textrm{D}_{T}^{\beta }\left ( \left \| D^{\alpha/2 }v  \right \|^{2}+\left \| v-\overline{u_{0}} \right \|^{2} \right )+\int _{\Omega }\rho  (v)\left | D^{\alpha }v \right |^{2}dx\\
     &\leq -\int _{\Omega }(\rho  (v)+v{\rho  }'(v))\left | D^{\alpha/2 }v \right |^{2}dx.
\end{align*}
Next, we notices that the first equation of \eqref{1.1} gives a variant form of this key identity:
$$ \partial _{t}^{\beta }(u-\overline{u_{0}})=-(-\Delta )^{\frac{\alpha}{2}}(\rho(v)u).$$
Furthermore, let $A=(-\Delta)^{\alpha/2}$ and $w(t)=u-\overline{u_{0}}$, we can deduce the following  formula from the above equation by fractional Duhamel type integral equation
 \begin{align*}
  \nonumber  w(t)&=E_{\beta }(-t^{\beta }A)w(\tau_{0})\\
&\quad-\int_{\tau_{0}}^{t}(t-s)^{\beta -1}E_{\beta ,\beta }(-(t-s)^{\beta }A)(-\Delta )^{\alpha /2}\left ( (\rho (v(s))-\rho_{0})u(s) \right )ds.
\end{align*}
Then, by Hardy-Little-wood-Sobolev inequality and Lemma \ref{lemma2.4}, the exponential stabilization of $(u,v)$ can be further acquired in $L^{\infty }\times W^{1,\infty }$.

As already mentioned above, in Appendix \ref{90oiuynytt}, we collect the definition of the fractional derivative, definition  of classical solution, definition  of fractional Laplacian and some useful Lemma  about  fractional derivative and fractional Laplacian. The Appendix \ref{ijugggrc} gives assumptions and properties about function $\rho(s)$ and optimization  proposition and lemma to proof of Theorem \ref{theorem4}.

\textbf{NOTATIONS:}  Throughout the paper, we fix Caputo derivative $\partial _{t}^{\beta },\beta\in (0,1)$ models memory effects in time. And we also replace fractional Laplacian $-(-\Delta )^{\frac{\alpha}{2}},\alpha\in(1,2)$ with  Neumann Laplacian operator $\Delta$.

\begin{theorem}\label{theorem4}
Suppose that $n\geq3$ and $\rho$ satisfies (H0),(H1) and (H3) in Appendix \ref{367897}. Then, for initial datum \eqref{90oiutgv},problem \eqref{1.1} possesses a unique global classical solution that is uniformly-in-time bounded.
In addition, let $A=(-\Delta)^{\alpha/2}$ and $w(t)=u-\overline{u_{0}}$, based on using fractional Duhamel type integral equation  to equation \eqref{4.17}, we can get 
   \begin{align*}
  \nonumber  w(t)&=E_{\beta }(-t^{\beta }A)w(\tau_{0})\\
&\quad-\int_{\tau_{0}}^{t}(t-s)^{\beta -1}E_{\beta ,\beta }(-(t-s)^{\beta }A)(-\Delta )^{\alpha /2}\left ( (\rho (v(s))-\rho_{0})u(s) \right )ds
\end{align*}

Then, there exist $\theta>0$ and $C>0$ depending on $u_{0},\rho,n$ and $\Omega$ such that, for all $t\geq 1,0<\beta<1$ makes
 \begin{equation}\label{o90875}
     \left \| u(\cdot ,t)-\overline{u_{0}} \right \|_{L^{\infty }(\Omega )}+\left \| v(\cdot ,t)-\overline{u_{0}} \right \|_{W^{1,\infty }(\Omega )}\leq Ce^{(-\theta)^{1/\beta}t},
 \end{equation}
where $\overline{u_{0}}=\frac{1}{\left | \Omega  \right |}\int _{\Omega }u_{0}dx.$
\end{theorem}
\section{Preliminaries}

\begin{definition}\textsuperscript{\cite{Kilbas2006}}
\begin{enumerate}[(i)]
    \item Assume that $X$ is a Banach space and let $u:\left [ 0,T \right ]\rightarrow X$. The Caputo fractional derivative operators of $u$ is defined by
\begin{eqnarray}
_{0}^{C}\textrm{D}_{t}^{\beta }u(t)&=&\frac{1}{\Gamma (1-\beta )}\int_{0}^{t}(t-s)^{-\beta }\frac{d}{ds}u(s)ds,\label{32}\\
\nonumber_{t}^{C}\textrm{D}_{T}^{\beta }u(t)&=&\frac{-1}{\Gamma (1-\beta )}\int_{t}^{T}(s-t)^{-\beta }\frac{d}{ds}u(s)ds,
\end{eqnarray}
 where $\Gamma (1-\beta )$ is the Gamma function. The above integrals are called the left-sided and the right-sided the Caputo fractional derivatives.
 \item For $u:[0,+\infty)\times \mathbb{R}^{n}\rightarrow \mathbb{R}$, the left Caputo fractional derivative with respect to time $t$ of $u$ is defined by
 \begin{equation}\label{44}
  \partial _{t}^{\beta }v(x,t)=\frac{1}{\Gamma (1-\beta )}\int_{0}^{t}\frac{\partial v}{\partial s}(x,s)(t-s)^{-\beta }ds,
\end{equation}
 \end{enumerate}
\end{definition}
\begin{definition}\textsuperscript{\cite{li2018some}}\label{li2018some}
     Let $B$ be a Banach space and $u\in L_{loc}^{1}((0,T);B)$. Let $u_{0}\in B$. We define the weak Caputo derivative of $u$ associated with initial data $u_{0}$ to be $\partial _{t}^{\beta }u\in \mathscr{D}'$ such that for any text function $\varphi \in C_{c}^{\infty }((-\infty,T );\mathbb{R}),$
     \begin{equation}
         \left \langle \partial _{t}^{\beta }u,\varphi   \right \rangle:=\int_{-\infty }^{T}(u-u_{0})\theta (t)(\tilde{\partial }_{T}^{\beta }\varphi )dt=\int_{0 }^{T}(u-u_{0})\tilde{\partial }_{T}^{\beta }\varphi dt
     \end{equation}
     where $\mathscr{D}'=\left \{ v\mid v: C_{c }^{\infty}((-\infty,T );\mathbb{R})\rightarrow B\,\text{is a bounded linear operator} \right \}.$
\end{definition}
\begin{definition}\textsuperscript{\cite{Kilbas2006}}\label{proposition 2.3}
  The Mittag-Leffler function $E_{\beta }(z),E_{\beta  ,\gamma  }(z)$ is defined as
   $$E_{\beta,1}(z)=E_{\beta}(z)=\sum_{k=0}^{\infty }\frac{z^{k}}{\Gamma (\beta k+1)},\quad E_{\beta,\gamma}(z)=\sum_{k=0}^{\infty }\frac{z^{k}}{\Gamma (\beta k+\gamma )},$$
  where $\gamma, z\in \mathbb{C},\mathbb{R}(\beta)>0,\mathbb{C}$ denote the complex plane.
  \end{definition}

\section{Time-Independent Upper Bounds of $v$}\label{89iuyh}
In the section,  we introduce a Lyapunov functional \eqref{3.1} to establish the time-independent upper bound for $v$ in higher dimensions. In the proof we will use Lemma \ref{lemma3.4},Proposition \ref{ijyuj90977} and Lemma \ref{lemmaiojd}, which are some simple optimization lemma that we state in Appendix \ref{0poiuytgb}.
\begin{lemma}\label{0oplkjh}
Assume that $n\geq 1,\alpha\in(1,2),\beta\in(0,1)$ and $\rho$ satisfies (H0), then for any $x\in \Omega,0<t<T_{max}$, it holds that
    \begin{equation}\label{2.1}
        _{0}^{C}\textrm{D}_{t}^{\beta }v+\rho  (v)u=(I+(-\Delta )^{\frac{\alpha}{2}} )^{-1}[\rho  (v)u].
    \end{equation}
Then, we can get the following inequality 
    \begin{equation}\label{2.2}
        v(x,t)\leq Cv_{0}(x)e^{\rho ^{1/\beta} (v_{*})t},
    \end{equation}
with $v_{0}\triangleq(I+(-\Delta )^{\frac{\alpha}{2}})^{-1}[u_{0}]$.
\end{lemma}
\begin{proof}
  First, since $\rho$ is non-increasing in $v$ and $v_{*}$ is expressed in Lemma \ref{lemma2.2},there holds
  $$\rho(v)\geq\rho(v_{*})$$
  for all $(x,t)\in [0,T_{max})$.
Next, applying the second equation  $(-\Delta )^{\frac{\alpha}{2}} v+v=u$, we have
    $$0\leq (I+(-\Delta )^{\frac{\alpha}{2}} )^{-1}[\rho  (v)u]\leq (I+(-\Delta )^{\frac{\alpha}{2}} )^{-1}[\rho  (v_{*})u]=\rho  (v_{*})v.$$
Thus, we obtain the following fractional differential inequality from \eqref{2.1} 
\begin{equation*}
        _{0}^{C}\textrm{D}_{t}^{\beta }v-\rho  (v_{*})v\leq -\rho  (v)u\leq 0,
    \end{equation*}
then, by Lemma\ref{Gao} and Lemma \ref{lemma390}, we get the  inequality
\begin{equation*}
        v(x,t)\leq Cv_{0}(x)e^{\rho ^{1/\beta} (v_{*})t},
    \end{equation*}
which $C$ is constant.  This completes the proof.
\end{proof}
\begin{remark}
    Thanks to the strictly positive time-independent lower bound $v_{*}$ of $v$ given in [\cite{fujie2020global},Lemma 3.2]. But the difference is that this paper uses fractional differential inequalities  to solve the equation \eqref{2.1} instead of Gronwall's inequality.
\end{remark}  
\begin{lemma}\label{lemma3.1}
   For any $\alpha\in(1,2),\beta\in(0,1)$ and $0\leq t<T_{max}$, it holds that
     \begin{equation}\label{3.1}
        \frac{1}{2}\tilde{\partial }_{T}^{\beta }\left ( \left \| D^{\alpha/2 }v  \right \|^{2}+\left \| v \right \|^{2} \right )+\int _{\Omega }\rho  (v)\left | D^{\alpha }v \right |^{2}dx+\int _{\Omega }(\rho (v)+v{\rho  }'(v))\left | D^{\alpha/2 }v \right |^{2}dx\leq 0.
    \end{equation}
Especially, under the assumption (H1) in Appendix \ref{367897}, there is $C>0$ that only depends on $u_{0}$ to get
    \begin{equation}\label{3.23}
        \underset{0\leq t<T_{max}}{\text{sup}}\left ( \left \| D^{\alpha/2 }v  \right \|^{2}+\left \| v \right \|^{2} \right )\leq C.
    \end{equation}
\end{lemma}
\begin{proof}
 First, we multiply the first equation of \eqref{1.1} by $v\in C_{c}^{\infty }(\Omega )$,integrating over $\Omega$ and replacing the second equation of \eqref{1.1} to get
    \begin{align*}
      & \left \langle \partial _{t}^{\beta }u,v \right \rangle=\int _{\Omega}(u-u_{0})\tilde{\partial }_{T}^{\beta }v dx=\int _{\Omega}u\tilde{\partial }_{T}^{\beta }v dx=\int _{\Omega}\left ( D^{\alpha }v+v \right )\tilde{\partial }_{T}^{\beta }vdx\\
&=\int _{\Omega} D^{\alpha }v\tilde{\partial }_{T}^{\beta }vdx+\int _{\Omega}v\tilde{\partial }_{T}^{\beta }vdx=\int _{\Omega} D^{\alpha/2 }vD^{\alpha/2 }(\tilde{\partial }_{T}^{\beta }v)dx+\int _{\Omega}v\tilde{\partial }_{T}^{\beta }vdx\\
&\geq \frac{1}{2}\int _{\Omega}\tilde{\partial }_{T}^{\beta }v^{2}dx+\frac{1}{2}\int _{\Omega}\tilde{\partial }_{T}^{\beta }\left |D^{\alpha/2 }v  \right |^{2}dx=\frac{1}{2}\tilde{\partial }_{T}^{\beta }\left ( \left \| D^{\alpha/2 }v  \right \|^{2}+\left \| v \right \|^{2} \right ).
    \end{align*}
Thanks to $v$ is non-decreases, then we obtain that
    \begin{align*}
        &\int _{\Omega }-v\left ( -\Delta  \right )^{\frac{\alpha }{2}}(\rho  (v)u)dx=-\int _{\Omega }\rho  (v)uD^{\alpha }vdx\\
&=-\int _{\Omega }\rho  (v)(v+D^{\alpha }v)D^{\alpha }vdx=-\int _{\Omega }\rho  (v)\left | D^{\alpha }v \right |^{2}dx-\int _{\Omega }\rho  (v)vD^{\alpha }vdx\\
&=-\int _{\Omega }\rho  (v)\left | D^{\alpha }v \right |^{2}dx-\int _{\Omega }D^{\alpha/2 }(\rho (v)v)D^{\alpha/2 }vdx\\
&= -\int _{\Omega }\rho  (v)\left | D^{\alpha }v \right |^{2}dx-\int _{\Omega }\left ( vD^{\alpha/2 }\rho  (v)+\rho  (v)D^{\alpha/2 }v \right )D^{\alpha/2 }vdx\\
&\qquad+\int _{\Omega }A_{\alpha /2}D^{\alpha/2 }v\int _{\Omega}\frac{(\rho  (v(x))-\rho  (v(y)))(v(x)-v(y))}{\left | x-y \right |^{n+\alpha }}dydx\\
&\leq -\int _{\Omega }\rho  (v)\left | D^{\alpha }v \right |^{2}dx-\int _{\Omega }(\rho  (v)+v{\rho  }'(v))\left | D^{\alpha/2 }v \right |^{2}dx.
    \end{align*}
Therefore, combining the above two inequalities, we have
    \begin{align}\label{789ijh}
   \nonumber     &\frac{1}{2}\tilde{\partial }_{T}^{\beta }\left ( \left \| D^{\alpha/2 }v  \right \|^{2}+\left \| v \right \|^{2} \right )\leq\left \langle \partial _{t}^{\beta }u,v \right \rangle\\
   &=\int _{\Omega }-v\left ( -\Delta  \right )^{\frac{\alpha }{2}}(\rho  (v)u)dx\leq-\int _{\Omega }\rho  (v)\left | D^{\alpha }v \right |^{2}dx-\int _{\Omega }(\rho  (v)+v{\rho  }'(v))\left | D^{\alpha/2 }v \right |^{2}dx.
    \end{align}
 Finally, since $D^{\alpha }v_{0}+v_{0}=u_{0}$ in $\Omega$ and $\partial _{v}v_{0}=0$ on $\partial \Omega $, by using Young's inequality, we infer that
    \begin{equation}
        \left \| D^{\alpha/2 }v_{0}  \right \|^{2}+\left \| v_{0} \right \|^{2}=\int _{\Omega }(v_{0}+D^{\alpha}v_{0})v_{0}dx= \int _{\Omega }u_{0}v_{0}dx\leq \frac{1}{2}\left \|  u_{0}\right \|^{2}+\frac{1}{2}\left \|  v_{0}\right \|^{2}.
    \end{equation}
Hence,
    $$\left \| D^{\alpha/2 }v_{0}  \right \|^{2}+\left \| v_{0} \right \|^{2}\leq 2\left \| D^{\alpha/2 }v_{0}  \right \|^{2}+\left \| v_{0} \right \|^{2}\leq \left \| u_{0} \right \|^{2}.$$
It can be seen that
     \begin{equation*}
        \underset{0\leq t<T_{max}}{\text{sup}}\left ( \left \| D^{\alpha/2 }v  \right \|^{2}+\left \| v \right \|^{2} \right )\leq C.
    \end{equation*}
This completes the proof.   
\end{proof}
\begin{remark}
    The proof of this lemma is mainly based on the properties of space-time fractional derivative.  By Definition \ref{li2018some} and Lemma \ref{lemma23}, system \eqref{1.1} possesses a Lyapunov function \eqref{3.1}, which implies a time-independent estimate of $\underset{t\geq0}{\text{sup}}\left \| v \right \|_{H^{1}(\Omega )}$ under the assumption of   \eqref{gyuij} and (H1).
\end{remark}

\begin{lemma}\label{lemma3.3}
Suppose that $n\geq 3,\alpha\in(1,2),\beta\in(0,1)$ and $\rho$ satisfies (H0) and (H2) in Appendix B.1. Then, for any $p>1+k,k>0$, there exist $\lambda _{1}>0$ and $\lambda _{2}>0$ independent of time such that
    \begin{align}\label{ijut567}
  \nonumber      &_{0}^{C}\textrm{D}_{t}^{\beta }\int _{\Omega }v^{p}dx+\lambda _{2}p\int _{\Omega }v^{p}dx+\frac{\lambda _{1}p(p-k-1)}{(p-k)^{2} }\int _{\Omega } \left |D^{\frac{\alpha}{2}} v^{\frac{p-k}{2} }  \right |^{2} dx+\lambda _{1}p\int _{\Omega }v^{p-k}dx\\
        &\leq 2\lambda _{2}p\int _{\Omega }v^{p}dx.
    \end{align}
\end{lemma}
\begin{proof}
First, mainly refer to \eqref{app3}-\eqref{3.7} in Appendix \ref{367897}, $\rho(\cdot )$ is non-increasing and for $s\geq s_{b},s_{b}>v^{*}$, we can get that 
   $$\footnote{Mainly refer to \eqref{app3} in Appendix \ref{367897}}1/\rho (s)\leq bs^{k},\qquad\footnote{Mainly refer to \eqref{app4} in Appendix \ref{367897}}1/\rho (s)\leq 1/\rho (s_{b}),$$
thus, for all $s\geq 0$, it obtain that
       \begin{equation}
       \footnote{Mainly refer to \eqref{3.7} in Appendix \ref{367897}} 1/\rho (s)\leq bs^{k}+1/\rho (s_{b}).
    \end{equation}
Next, multiplying both sides of \eqref{2.1} by $v^{p-1}$ with $p>k+1$, and using Lemma \ref{lemma1}, we get
\begin{align}
 \nonumber &  \frac{1}{p}\, _{0}^{C}\textrm{D}_{t}^{\beta }\int _{\Omega }v^{p}dx+\int _{\Omega }v^{p-1}\rho  (v)udx\\
&\leq  \int _{\Omega }\,v^{p-1}\, _{0}^{C}\textrm{D}_{t}^{\beta }vdx+\int _{\Omega }v^{p-1}\rho  (v)udx=\int _{\Omega} v^{p-1}(I+(-\Delta )^{\frac{\alpha}{2}} )^{-1}[\rho  (v)u]dx
\end{align}
By the comparison principle of fractional elliptic equation and $\rho(v)\leq \rho (v_{*})$, we deduce that
$$(I+(-\Delta )^{\frac{\alpha}{2}} )^{-1}[\rho  (v)u]\leq \rho (v_{*})v,$$
and thus
$$\int _{\Omega} v^{p-1}(I+(-\Delta )^{\frac{\alpha}{2}} )^{-1}[\rho  (v)u]dx\leq \rho (v_{*})\int _{\Omega} v^{p} dx.$$
Based on equation \eqref{app5}-\eqref{8hfhjn} in Appendix B.1, one has
  $$\footnote{Mainly refer to \eqref{app5} in Appendix \ref{367897}} \int _{\Omega }v^{p-1}\rho  (v)udx\geq\int _{\Omega }v^{p-1}(bv^{k}+1/\rho (s_{b}))^{-1}udx\geq C_{1}\int _{\Omega }(v_{k}+1)^{-1}v^{p-1}udx .$$
  where $1/C_{1}=\frac{max\left \{ b\rho (s_{b}),1 \right \}}{\rho (s_{b})}>0$ in Appendix B.1,taking into account the fact that
\begin{equation*}
   \footnote{Mainly refer to \eqref{app6} in Appendix \ref{367897}}  bv^{k}+1/\rho (s_{b})=\frac{1}{\rho (s_{b})} (bv^{k}\rho (s_{b})+1)\leq \frac{max \left \{  b\rho (s_{b}),1\right \} }{\rho (s_{b})}(v^{k}+1).
\end{equation*}
Thanks to $v^{k}\geq v_{*}^{k}$ by Lemma \ref{lemma2.2}, it holds that
\begin{equation*}
   \footnote{Mainly refer to \eqref{78uuyh} in Appendix \ref{367897}}  (v^{k}+1)^{-1}v^{p-1} \geq (v^{k}+v_{*}^{-k}v^{k})^{-1} v^{p-1}=\frac{v^{p-k-1} }{1+v_{*}^{-k} }.
\end{equation*}
Thus, we have
\begin{equation*}
     \footnote{Mainly refer to \eqref{8hfhjn} in Appendix \ref{367897}}  \int _{\Omega }v^{p-1}\rho  (v)udx\geq C_{2}\int _{\Omega }v^{p-k-1}udx,
\end{equation*}
which $C_{2}>0$ may depend on initial data $n,\Omega,\rho$, but is independent of $p$ and time. Replacing with $v+(-\Delta )^{\alpha /2}v=u$ and using Lemma \ref{lemmawer}, we can obtain
\begin{align*}
    \int _{\Omega }v^{p-k-1}udx&=\int _{\Omega }v^{p-k-1}(v+(-\Delta )^{\frac{\alpha}{2}} v)dx\\
&=\int _{\Omega }v^{p-k}dx+ \int _{\Omega }v^{p-k-1}D^{\alpha}vdx\\
&\geq\int _{\Omega }v^{p-k}dx+\frac{4(p-k-1)}{(p-k)^{2} } \int _{\Omega }\left | D^{\frac{\alpha}{2}} v^{\frac{p-k}{2} } \right |^{2} dx.
\end{align*}
Finally, we can get
$$ _{0}^{C}\textrm{D}_{t}^{\beta }\int _{\Omega }v^{p}dx+\frac{\lambda _{1}p(p-k-1)}{(p-k)^{2} }\int _{\Omega } \left | D^{\frac{\alpha}{2}} v^{\frac{p-k}{2} }  \right |^{2} dx+\lambda _{1}p\int _{\Omega }v^{p-k}dx \leq \lambda _{2}p\int _{\Omega }v^{p}dx   $$
where $\lambda _{1}>0,\lambda _{2}>0$ independent of $p$ and time.To get inequality \eqref{ijut567}, we need to add $\lambda _{2}p\int _{\Omega }v^{p}dx$ on both sides of the above inequality.
\end{proof}
\begin{remark}
    From \eqref{3.4} and \eqref{3.5},  we proceed to derive the uniform-in-time upper bounds of $v$ which based on a delicate Alikako-Moser iteration \cite{alikakos1979application}. However, in the process of proof, compared with the literature \cite{jiang2022boundedness}, we use  Stroock-Varopoulos’ inequality to shrink.
\end{remark}

 \section{Proof of Theorem \ref{theorem4}}
\subsection{Uniform-in-time boundedness}\label{iokjuyhnnj}
In order to proof of Theorem \ref{theorem4},  the first result of the section will establish an estimate involving the weighted energy $\int _{\Omega }u^{p+1}\rho^{q}(v)dx$ for some $1+p>\frac{n}{2},q>0$ with the help of time-independent upper bound of $v$ in hand. Then, we will use Appendix \ref{inequaliyy}-Appendix \ref{i9io87h}, which are  some useful lemma about time-space fractional derivative.
\begin{lemma}\label{897hhg}
For any $\alpha\in(1,2),\beta\in(0,1),0\leq t<T_{max}$ and $p,q>0$, it holds that
    \begin{align}\label{4.1}
 \nonumber   &_{0}^{C}\textrm{D}_{t}^{\beta }\int _{\Omega }u^{p+1}\rho^{q}(v)dx+p(p+1)\int _{\Omega }u^{p-1}\rho  ^{q+1}(v)\left | D^{\alpha/2}u \right |^{2}dx\\
  \nonumber   &\qquad+q\int _{\Omega }\left ( (q+p+1) \left | {\rho }'(v) \right |^{2}+\rho  {\rho  }''\right )u^{p+1}\rho  ^{q-1}(v) D^{\alpha/2}vdx\\
 \nonumber    &\qquad -q\int _{\Omega }u^{p+1}\rho  ^{q-1}{\rho  }'(v)(I+(-\Delta)^{\frac{\alpha}{2}} )^{-1}[u\rho  (v) ]dx\\
  \nonumber     &\leq -q\int _{\Omega }u^{p+1}\rho  ^{q}(v){\rho  }'(v)vdx\\
  &\qquad-(2q+p)(p+1)\int _{\Omega }u^{p}\rho  ^{q}(v){\rho  }'(v)D^{\alpha/2}uD^{\alpha/2}vdx.
\end{align}
\end{lemma}
\begin{proof}
For any $p,q>0$, by Lemma \ref{lemma1}-Lemma \ref{lemma0099}, we can infer the following inequality
    \begin{align}\label{4.2}
   \nonumber    & _{0}^{C}\textrm{D}_{t}^{\beta }\int _{\Omega }u^{p+1}\rho  ^{q}(v)dx\\
    \nonumber     &\leq \int _{\Omega }u^{p+1}\,_{0}^{C}\textrm{D}_{t}^{\beta }\rho  ^{q}(v)dx+\int _{\Omega }\rho  ^{q}(v)_{0}^{C}\textrm{D}_{t}^{\beta }u^{p+1}dx\\
  \nonumber     &\quad  +\frac{t^{-\beta }}{\Gamma (1-\beta )}\left ( u^{p+1}\rho  ^{q}(0)+\rho  ^{q}u^{p+1}(0)-u^{p+1}(0)\rho  ^{q}(0) \right )\\
\nonumber &\leq (p+1)\int _{\Omega }\rho  ^{q}(v)u^{p}\,_{0}^{C}\textrm{D}_{t}^{\beta }udx+q\int _{\Omega }u^{p+1}\rho  ^{q-1}{\rho  }'(v)_{0}^{C}\textrm{D}_{t}^{\beta }vdx\\
\nonumber &\leq (p+1)\int _{\Omega }\rho  ^{q}(v)u^{p}\,_{0}^{C}\textrm{D}_{t}^{\beta }udx+q\int _{\Omega }u^{p+1}\rho  ^{q-1}{\rho  }'(v)( (I+(-\Delta)^{\frac{\alpha}{2}} )^{-1}\\
\nonumber&\qquad [u\rho  (v) ]-u\rho  (v))dx\\
\nonumber &\leq (p+1)\int _{\Omega }\rho  ^{q}(v)u^{p}\,_{0}^{C}\textrm{D}_{t}^{\beta }udx+q\int _{\Omega }u^{p+1}\rho  ^{q-1}{\rho  }'(v) (I+(-\Delta)^{\frac{\alpha}{2}} )^{-1}[u\rho  (v) ]dx\\
&\qquad-q\int _{\Omega }u^{p+1}\rho  ^{q}(v){\rho  }'(v)udx.
    \end{align}
Replacing the right item of formula \eqref{4.2} with $v+(-\Delta )^{\frac{\alpha }{2}}v=u$, then we have 
    \begin{align}\label{29086}
  \nonumber    &_{0}^{C}\textrm{D}_{t}^{\beta }\int _{\Omega }u^{p+1}\rho  ^{q}(v)dx- (p+1)\int _{\Omega }\rho  ^{q}(v)u^{p}\,_{0}^{C}\textrm{D}_{t}^{\beta }udx+q\int _{\Omega }u^{p+1}\rho  ^{q}(v){\rho  }'(v)(-\Delta)^{\frac{\alpha}{2}} vdx\\
   \nonumber   &\qquad-q\int _{\Omega }u^{p+1}\rho  ^{q-1}{\rho  }'(v)(I+(-\Delta)^{\frac{\alpha}{2}} )^{-1}[u\rho  (v) ]dx\\
      &\leq -q\int _{\Omega }u^{p+1}\rho  ^{q}(v){\rho  }'(v)vdx.
    \end{align}
By \eqref{556678}, \eqref{33.11},\eqref{gyuij} and the first equation of \eqref{1.1},  we infer that
\begin{align}\label{4.4e3}
 \nonumber   & -(p+1)\int _{\Omega }\rho  ^{q}(v)u^{p}\,_{0}^{C}\textrm{D}_{t}^{\beta }udx\\
 \nonumber   &= (p+1)\int _{\Omega }\rho  ^{q}(v)u^{p}(-\Delta )^{\frac{\alpha }{2}}(\rho  (v)u)dx\\
 \nonumber   &= (p+1)\int _{\Omega }D^{\alpha/2}(\rho  ^{q}(v)u^{p})D^{\alpha/2}(\rho  (v)u)dx\\
 \nonumber   &= (p+1)\int _{\Omega }\left ( \rho (v)D^{\alpha/2}u+uD^{\alpha/2}\rho (v)-\mathcal{N}(\rho (v),u) \right )\\
 \nonumber&\qquad\left ( \rho ^{q}(v)D^{\alpha/2}u^{p}+u^{p}D^{\alpha/2}\rho  ^{q}(v)-\mathcal{N}(\rho ^{q}(v),u^{p})\right )dx\\
 \nonumber   &\leq (p+1)\int _{\Omega }\left ( \rho (v)D^{\alpha/2}u+u{\rho  }'(v)D^{\alpha/2}v  \right )\\
\nonumber &\qquad\left ( p\rho  ^{q}u^{p-1}D^{\alpha/2}u+qu^{p}\rho  ^{q-1}(v){\rho  }'(v)D^{\alpha/2}v \right )dx\\
 \nonumber   &\leq p(p+1)\int _{\Omega }u^{p-1}\rho  ^{q+1}(v)\left | D^{\alpha/2}u \right |^{2}dx+q(p+1)\int _{\Omega }u^{p+1}\rho  ^{q-1}(v)\left | {\rho  }'(v) \right |^{2}\left | D^{\alpha/2}v \right |^{2}dx\\
&\qquad+(p+1)(p+q)\int _{\Omega }u^{p}\rho  ^{q}(v){\rho  }'(v)D^{\alpha/2}uD^{\alpha/2}vdx.
\end{align}
where $$\mathcal{N}(u,v)=A_{\alpha /2}\int _{\Omega}\frac{\left ( (u(x)-u(y))(v(x)-v(y)) \right )}{\left | x-y \right |^{n+\alpha }}dy,A_{\alpha /2}>0.$$
And we also obtain
\begin{align}\label{389jy}
  \nonumber&  q\int _{\Omega }u^{p+1}\rho^{q}(v){\rho  }'(v)(-\Delta)^{\frac{\alpha}{2}} vdx\\
\nonumber&=q\int _{\Omega }D^{\alpha/2 }\left ( u^{p+1}\rho  ^{q}(v){\rho  }'(v) \right )D^{\alpha/2 }vdx\\
\nonumber&\leq q(p+1)\int _{\Omega } u^{p}\rho  ^{q}(v){\rho  }'(v)D^{\alpha/2 }uD^{\alpha/2 }vdx+q^{2}\int _{\Omega }u^{p+1}\rho  ^{q-1}(v)\left | {\rho  }'(v) \right |^{2}\left | D^{\alpha/2 }v \right | ^{2}dx\\
\nonumber&\qquad+q\int _{\Omega }{\rho  }''(v)u^{p+1}\rho  ^{q}(v)\left | D^{\alpha/2 }v \right | ^{2}dx-q\int _{\Omega }D^{\alpha/2 }v\mathcal{N}(u^{p+1},\rho  ^{q}(v),{\rho  }'(v))dx\\
\nonumber&\leq q(p+1)\int _{\Omega } u^{p}\rho  ^{q}(v){\rho  }'(v)D^{\alpha/2 }uD^{\alpha/2 }vdx+q^{2}\int _{\Omega }u^{p+1}\rho  ^{q-1}(v)\left | {\rho  }'(v) \right |^{2}\left | D^{\alpha/2 }v \right | ^{2}dx\\
&\qquad+q\int _{\Omega }{\rho  }''(v)u^{p+1}\rho  ^{q}(v)\left | D^{\alpha/2 }v \right | ^{2}dx.
\end{align}
Instead \eqref{4.4e3} and \eqref{389jy} to \eqref{29086}, then we can get \eqref{4.1}.
\end{proof}

\begin{lemma}\label{lemma4.3}
  Suppose that $n\geq 3,\alpha\in(1,2),\beta\in(0,1)$ and $\rho  \left ( \cdot  \right )$ satisfies (H0),(H1) and (H3) in Appendix \ref{367897}. For any $1+p\in(1,l_{0}^{2})$, and $q=\frac{pl_{0}}{2}$,it holds that
  \begin{equation*}
     _{0}^{C}\textrm{D}_{t}^{\beta }\int _{\Omega }u^{p+1}\rho  ^{q}(v)dx+C_{6}\int _{\Omega }u^{p+1}\rho  ^{q}(v)dx\leq C_{7}\left ( \int _{\Omega }u^{\frac{1+p}{2}}dx \right )^{2},
\end{equation*}
with $C_{6}>0,C_{7}>0$. Then there exist $p>\frac{n}{2}-1$ and $C>0$ independent of time such that
    $$\underset{0\leq t<T_{max}}{\text{sup}}\int _{\Omega }u^{1+p}dx\leq C.$$
\end{lemma}
\begin{proof}
 By Young's inequality, let $p=q=2,\varepsilon =\frac{2p(\delta_{0}-1)}{p+2q}$,$a=u^{\frac{p-1}{2}}\rho  ^{\frac{q+1}{2}}D^{\alpha /2}u, b=u^{\frac{p+1}{2}}\rho  ^{\frac{q-1}{2}}D^{\alpha /2}v$, we have
    \begin{align}
       \nonumber & -(2q+p)(p+1)\int _{\Omega }u^{p}\rho  ^{q}(v){\rho  }'(v)D^{\alpha/2}uD^{\alpha/2}vdx\\
\nonumber &\leq (p+1)p(1-\delta _{0})\int _{\Omega }u^{p-1}\rho  ^{q+1}\left | D^{\alpha /2}u \right |^{2}dx\\
&\qquad+\frac{(p+1)(p+2q)^{2}}{4p(1-\delta _{0})}\int _{\Omega }u^{p+1}\rho  ^{q-1}\left | {\rho  }' \right |^{2}\left | D^{\alpha /2}v \right |^{2}dx,
    \end{align}
with $q,\delta_{0}$ chosen in inequality \eqref{lemma4.2}.
Then it follows from \eqref{4.1} and \eqref{lemma4.2} that
    \begin{align}
     \nonumber & _{0}^{C}\textrm{D}_{t}^{\beta }\int _{\Omega }u^{p+1}\rho  ^{q}(v)dx+\delta _{0}p(p+1)\int _{\Omega }u^{p-1}\rho  ^{q+1}\left | D^{\alpha /2}u \right |^{2}dx\\
 \nonumber &\qquad-q\int _{\Omega }u^{p+1}\rho  ^{q-1}{\rho  }'(v)(I+(-\Delta)^{\frac{\alpha}{2}} )^{-1}[u\rho  (v) ]dx\\
&\leq -q\int _{\Omega }u^{p+1}\rho  ^{q}(v){\rho  }'(v)vdx
    \end{align}
with any $1+p\in(1,l_{0}^{2})$, and $q=\frac{pl_{0}}{2}$. 

Due to $v_{*}\leq v\leq v^{*}$, for any $1+p\in(1,l_{0}^{2})$, and $q=\frac{pl_{0}}{2}$, then it makes
     \begin{equation}\label{4.9}
         _{0}^{C}\textrm{D}_{t}^{\beta }\int _{\Omega }u^{p+1}\rho  ^{q}(v)dx+C_{3}\int _{\Omega }u^{p-1}\rho  ^{q+1}\left | D^{\alpha /2}u \right |^{2}dx\leq {C}'_{3}\int _{\Omega }u^{p+1}dx.
     \end{equation}
where $C_{3},{C}'_{3}$ depending on $p,\delta_{0},\rho$. Using the Nash-Gagliardo-Nirenberg-Type Inequality(Lemma \ref{lemmagn}) and Young's inequality, due to the uniform-in-time boundedness of $\rho ^{q}(v)$, we infer that
     $$\int _{\Omega }u^{p+1}\rho  ^{q}(v)dx\leq C_{4}\int _{\Omega }u^{1+p}dx={C}_{4}^{'} \left \| \xi  \right \|_{L^{2}(\Omega )}^{2}\leq \varepsilon \left \|  D^{\alpha/2}\xi \right \|^{2}+C_{5}(\varepsilon )\left \| \xi  \right \|_{L^{1}(\Omega )}^{2}$$
  with $\varepsilon >0,\xi =u^{\frac{1+p}{2}}$ and $C_{4},{C}_{4}^{'},C_{5}>0$ independent of time. Thus, by choosing proper small $\varepsilon >0$ and $1+p\in(1,l_{0}^{2})$, and $q=\frac{pl_{0}}{2}$, we can obtain
\begin{equation}\label{4.11}
     _{0}^{C}\textrm{D}_{t}^{\beta }\int _{\Omega }u^{p+1}\rho  ^{q}(v)dx+C_{6}\int _{\Omega }u^{p+1}\rho  ^{q}(v)dx\leq C_{7}\left ( \int _{\Omega }u^{\frac{1+p}{2}}dx \right )^{2},
\end{equation}
with $C_{6}>0,C_{7}>0$ independent of time. Then we first solve the above fractional differential inequality, and then from $\left \| u(\cdot ,t) \right \|_{L^{1}(\Omega )}=\left \| u_{0} \right \|_{L^{1}(\Omega )}$ and the uniform-in-time lower and upper boundedness of $\rho^{q}(v)$, we can repeatedly deduce \eqref{4.11}, such that
\begin{equation}\label{4.12}
    \underset{t\geq 0}{\text{sup}}\int _{\Omega }u^{1+p}dx\leq C_{8}.
\end{equation}
Finally, for any $n\geq 3$, we can obtain
$$\frac{n}{2}< \left ( \frac{n+2}{4} \right )^{2}< l_{0}^{2}.$$
Therefore, we can find $p>0$ satisfying $1+p>\frac{n}{2}$ such that \eqref{4.11}-\eqref{4.12} holds.
\end{proof}
\begin{remark}
    According to the proof method of [\cite{jiang2022boundedness},Lemma 4.3], the formula in the Lemma \ref{897hhg} is used by using Young’s inequality. But unlike, we combine  Nash-Gagliato-Nirenberg-Type inequality and fractional differential inequality to get the uniform-in-time boundedness of solutions.
\end{remark}

\subsection{Exponential stabilization toward constant steady states}\label{90oiuytgv}
In the section, we would like to prove that the exponential stabilization of the global solutions by using Lyapunov functional \eqref{3.1} and fractional Duhamel type integral equation to \eqref{4.17}.
\begin{lemma}\label{lemma4.5}
  Suppose that $\alpha\in(1,2),\beta\in(0,1)$, for any $p>0,C_{13}>0$, it holds that
  \begin{align}
  \nonumber  &_{0}^{C}\textrm{D}_{t}^{\beta }\int _{\Omega }\left | u-\overline{u_{0}} \right |^{p}dx+\alpha_{2}\int _{\Omega }\left | u-\overline{u_{0}} \right |^{p}dx\\
  &\leq C_{13}\left ( \int _{\Omega }\left |D^{\alpha/2 }v  \right |^{2}dx+ \int _{\Omega }\left | u-\overline{u_{0}} \right |^{2}dx\right ).
\end{align}
Then, there exist constants $w>0$ and $0<\beta<1$ such that
\begin{equation}
    \left \| v-\overline{u_{0}} \right \|_{W^{1,\infty }(\Omega )}\leq Ce^{(-w)^{1/\beta}t},\qquad \forall t>0.
\end{equation}
where $C>0$ depending on $u_{0},\rho,n$ and $\Omega$.
\end{lemma}
\begin{proof}
It is already known that there is $\overline{u}(t)=\overline{v}(t)=\overline{u_{0}}$. Meanwhile,it follows that $_{0}^{C}\textrm{D}_{t}^{\beta }\int _{\Omega }v(t)dx=0$ and we obtain that
    \begin{align*}
        _{0}^{C}\textrm{D}_{t}^{\beta }\left \| v(\cdot ,t)-\overline{u_{0}} \right \|^{2}&=\, _{0}^{C}\textrm{D}_{t}^{\beta }\int _{\Omega }\left ( v^{2}-2\overline{u_{0}}v+\overline{u_{0}}^{2} \right )dx\\
&=\,_{0}^{C}\textrm{D}_{t}^{\beta }\int _{\Omega }v^{2}dx-2\overline{u_{0}}\,_{0}^{C}\textrm{D}_{t}^{\beta }\int _{\Omega }vdx\\
&=\,_{0}^{C}\textrm{D}_{t}^{\beta }\left \| v(\cdot ,t) \right \|^{2}.
    \end{align*}
Thus, we can infer from \eqref{3.1} that
\begin{align}\label{4.14}
\nonumber     &\frac{1}{2}\,_{0}^{C}\textrm{D}_{T}^{\beta }\left ( \left \| D^{\alpha/2 }v  \right \|^{2}+\left \| v-\overline{u_{0}} \right \|^{2} \right )+\int _{\Omega }\rho  (v)\left | D^{\alpha }v \right |^{2}dx\\
     &\leq -\int _{\Omega }(\rho  (v)+v{\rho  }'(v))\left | D^{\alpha/2 }v \right |^{2}dx.
\end{align}
Because of $v\leq v^{*}$ and $\rho$ is non-increasing, we obtain that
$$\int _{\Omega }\rho  (v)\left | D^{\alpha }v \right |^{2}dx\geq \rho  (v^{*})\int _{\Omega }\left | D^{\alpha }v \right |^{2}dx,$$
 further, we use Hardy-Little-wood-Sobolev Inequality $\mu_{1}\left \| v-\overline{u_{0}} \right \|^{2}\leq \left \| D^{\alpha/4 }v \right \|^{2}$, yields that
\begin{align}
  \nonumber  \int _{\Omega }\rho  (v)\left | D^{\alpha/2 }v \right |^{2}dx&\geq \rho  (v^{*})\int _{\Omega }\left | D^{\alpha }v \right |^{2}dx\\
    &\geq\frac{\mu_{1}\rho  (v^{*})}{1+\mu_{1}}\left ( \left \| D^{\alpha/2 }v  \right \|^{2}+\left \| v-\overline{u_{0}} \right \|^{2} \right ),
\end{align}
where $\mu_{1}>0$. If $\partial_{v} v=0$ on $\partial \Omega $, by using $ \mu _{1}\left \| D^{\alpha/2 } v \right \|^{2}\leq \left \| D^{\alpha} v \right \|^{2}$ and \eqref{4.14},we will infer that
$$\frac{1}{2}\,_{0}^{C}\textrm{D}_{T}^{\beta }\left ( \left \| D^{\alpha/2 }v  \right \|^{2}+\left \| v-\overline{u_{0}} \right \|^{2} \right )+\frac{\mu_{1}\rho  (v^{*})}{1+\mu_{1}}\left ( \left \| D^{\alpha/2 }v  \right \|^{2}+\left \| v-\overline{u_{0}} \right \|^{2} \right )\leq 0,$$
which by standard Lemma \ref{Gao} and Lemma \ref{lemma390} analysis yields that, for all $t\geq 0,$
\begin{equation}
   \left \| D^{\alpha/2 }v  \right \|^{2}+\left \| v-\overline{u_{0}} \right \|^{2}\leq \left ( \left \| D^{\alpha/2 }v_{0}  \right \|^{2}+\left \| v_{0}-\overline{u_{0}} \right \|^{2} \right )e^{\left ( -\frac{2\mu_{1}\rho  (v^{*})}{1+\mu_{1}} \right )^{1/\beta }t}
\end{equation}
where $v_{0}=\left ( I+(-\Delta )^{1/\alpha } \right )^{-1}\left [ u_{0} \right ].$

Next, from the first equation of \eqref{1.1}, we have that
\begin{equation}\label{4.17}
     \partial _{t}^{\beta }\left ( u-\overline{u_{0}} \right )=-(-\Delta )^{\frac{\alpha}{2}}(\rho  (v)u).
\end{equation}
Multiplying \eqref{4.17} by $u-\overline{u_{0}}$, by Lemma \ref{lemma1} and \eqref{3.133}, we obtain that
\begin{align*}
    \frac{1}{2}\,_{0}^{C}\textrm{D}_{t}^{\beta }\int _{\Omega }\left | u-\overline{u_{0}} \right |^{2}dx&\leq -\int _{\Omega }\left ( u-\overline{u_{0}} \right )D^{\alpha }(u\rho  (v))dx\\
&=-\int _{\Omega }D^{\alpha/2 }\left ( u-\overline{u_{0}} \right )D^{\alpha/2 }(u\rho  (v))dx\\
&=-\int _{\Omega }D^{\alpha/2 }uD^{\alpha/2 }(u\rho  (v))dx\\
&= -\int _{\Omega }D^{\alpha/2 }u\left ( u{\rho  }'(v)D^{\alpha/2 }v+\rho  (v) D^{\alpha/2 }v\right )\\
&+\int _{\Omega }A_{\alpha /2}D^{\alpha/2 }u\int _{\Omega}\frac{(\rho (v(x))-\rho  (v(y)))(u(x)-u(y))}{\left | x-y \right |^{n+\alpha }}dydx\\
&\leq -\int _{\Omega }D^{\alpha/2 }u\left ( u{\rho }'(v)D^{\alpha/2 }v+\rho  (v) D^{\alpha/2 }v\right ),
\end{align*}
where $u$ is non-decreasing. Thus, we can get
\begin{align*}
     &\frac{1}{2}\,_{0}^{C}\textrm{D}_{t}^{\beta }\int _{\Omega }\left | u-\overline{u_{0}} \right |^{2}dx+\int _{\Omega }\rho  (v) \left |D^{\alpha/2 }u  \right |^{2}dx \\
     &\leq-\int _{\Omega }u{\rho  }'(v)D^{\alpha/2 }vD^{\alpha/2 }udx\\
     &\leq \frac{1}{2}\int _{\Omega }\rho  (v)\left | D^{\alpha/2 }v \right |^{2}dx+\frac{1}{2}\int _{\Omega }\frac{u^{2}\left | {\rho  }'(v) \right |^{2}}{\rho  (v)}\left | D^{\alpha/2 }u \right |^{2}dx.
\end{align*}
Due to $u$ and $v$ that are both uniformly-in-time bounded, we have
\begin{equation}\label{0poiu}
    _{0}^{C}\textrm{D}_{t}^{\beta }\int _{\Omega }\left | u-\overline{u_{0}} \right |^{2}dx+\rho (v^{*})\int _{\Omega } \left |D^{\alpha/2 }u  \right |^{2}dx\leq C_{9} \int _{\Omega } \left |D^{\alpha/2 }v  \right |^{2}dx,
\end{equation}
which time -independent $C_{9}>0$ is constant. Next, by using  Hardy-Little-wood-Sobolev Inequality, one can find a constant $0<w_{1}<-\frac{2\mu_{1}\rho  (v^{*})}{1+\mu_{1}}$ such that
\begin{equation}\label{o90okjh}
    _{0}^{C}\textrm{D}_{t}^{\beta }\int _{\Omega }\left | u-\overline{u_{0}} \right |^{2}dx+w_{1}\int _{\Omega }\left | u-\overline{u_{0}} \right |^{2}dx\leq C_{10}\int _{\Omega }\left |D^{\alpha/2 }v  \right |^{2}dx.
\end{equation}
where $C_{10}>0$ depending only on initial datum, $\rho,n,\Omega$. Solving the above fractional differential inequality by Lemma \ref{lemma2.14} yields that
\begin{equation}\label{89076}
    \left \| u-\overline{u_{0}} \right \|^{2}\leq \left \| u_{0}-\overline{u_{0}} \right \|^{2}+\frac{C_{10}}{\Gamma (\beta )}\int_{0}^{t}(t-s)^{\beta -1}\left \| D^{\alpha/2 }v(s)  \right \|^{2}ds.
\end{equation}
Applying Grownwall type of lemma in \eqref{89076}, we obtain that
\begin{equation*}
    \left | u-\overline{u_{0}} \right |\leq \left \| u_{0}-\overline{u_{0}} \right \|^{2}\text{exp}\left ( \frac{C_{10}}{\Gamma (\beta )}\int_{0}^{t}(t-s)^{\beta -1}\left \| D^{\alpha/2 }v(s)  \right \|^{2}ds \right ).
\end{equation*}

Next, for any $p>2$, multiplying \eqref{4.17} by $\left | u-\overline{u_{0}} \right |^{p-2}(u-\overline{u_{0}})$ to obtain that
\begin{align*}
    &\frac{1}{p}\,_{0}^{C}\textrm{D}_{t}^{\beta }\int _{\Omega }\left | u-\overline{u_{0}} \right |^{p}dx+(p-1)\int _{\Omega }\rho  (v)\left | u-\overline{u_{0}} \right |^{p-2} \left |D^{\alpha/2 }u  \right |^{2}dx\\
&\leq -(p-1)\int _{\Omega }u\left | u-\overline{u_{0}} \right |^{p-2}{\rho  }'(v)D^{\alpha/2 }vD^{\alpha/2 }udx\\
&\leq \frac{p-1}{2}\int _{\Omega }\rho  (v)\left | u-\overline{u_{0}} \right |^{p-2}\left | D^{\alpha/2 }v \right |^{2}dx\\
&\qquad+\frac{p-1}{2}\int _{\Omega }\frac{u^{2}\left | {\rho  }'(v) \right |^{2}}{\rho  (v)}\left | u-\overline{u_{0}} \right |^{p-2}\left | D^{\alpha/2 }u \right |^{2}dx.
\end{align*}
Similarly, there are constants $C_{11}=C_{11}(p)>0$ independent of time and $w_{2}<w_{1}$ such that
 \begin{align*}
     &_{0}^{C}\textrm{D}_{t}^{\beta }\int _{\Omega }\left | u-\overline{u_{0}} \right |^{p}dx+\alpha_{2}\int _{\Omega }\left | u-\overline{u_{0}} \right |^{p}dx\\
&\leq C_{11}\int _{\Omega }\left |D^{\alpha/2 }v  \right |^{2}dx+\alpha_{2}\int _{\Omega }\left | u-\overline{u_{0}} \right |^{p}dx.
 \end{align*}
\footnote{The inequality here is mainly referred to [\cite{jiang2022boundedness},Lemma 4.5], which depends on the boundary of $\left \| u-\overline{u_{0}} \right \|_{L^{\infty }(\Omega )}^{p-2}$.}Observing that
$$\int _{\Omega }\left | u-\overline{u_{0}} \right |^{p}dx\leq \left \| u-\overline{u_{0}} \right \|_{L^{\infty }(\Omega )}^{p-2}\int _{\Omega }\left | u-\overline{u_{0}} \right |^{2}dx\leq C_{12}\int _{\Omega }\left | u-\overline{u_{0}} \right |^{2}dx,$$
we can get
\begin{align}
  \nonumber  &_{0}^{C}\textrm{D}_{t}^{\beta }\int _{\Omega }\left | u-\overline{u_{0}} \right |^{p}dx+\alpha_{2}\int _{\Omega }\left | u-\overline{u_{0}} \right |^{p}dx\\
  &\leq C_{13}\left ( \int _{\Omega }\left |D^{\alpha/2 }v  \right |^{2}dx+ \int _{\Omega }\left | u-\overline{u_{0}} \right |^{2}dx\right ),
\end{align}
In the view of \eqref{89076}, we have
\begin{align}\label{4.20}
   \nonumber  \left \| u-\overline{u_{0}} \right \|^{p}&\leq \left \| u_{0}-\overline{u_{0}} \right \|^{p}\\
     &+\frac{C_{13}}{\Gamma (\beta )}\int_{0}^{t}(t-s)^{\beta -1}\left ( \left \| D^{\alpha/2 }v(s)  \right \|^{2}+\left \| u-\overline{u_{0}} \right \|^{2} \right )ds=:M
\end{align}
and
\begin{align*}
     &\left | u-\overline{u_{0}} \right |^{p}\\
     &\leq \left \| u_{0}-\overline{u_{0}} \right \|^{p}\text{exp}\left ( \frac{C_{13}}{\Gamma (\beta )}\int_{0}^{t}(t-s)^{\beta -1}\left ( \left \| D^{\alpha/2 }v(s)  \right \|^{2}+\left \| u-\overline{u_{0}} \right \|^{2} \right )ds \right ).
\end{align*}
Finally, we note that from the second equation of \eqref{1.1}
$$\left ( v-\overline{u_{0}} \right )+(-\Delta )^{\frac{\alpha}{2}} \left ( v-\overline{u_{0}} \right )=u-\overline{u_{0}}.$$
\footnote{It mainly depends on the elliptic regularity and Sobolev embeddings mentioned in the literature \cite{jiang2022boundedness} to verify.}Choosing some $p_{0}>n$ in \eqref{4.20}, one may deduce, by elliptic regularity and Sobolev embeddings, that
 $$\left \| v-\overline{u_{0}} \right \|_{W^{1,\infty }(\Omega )}\leq C_{15}\left \| v-\overline{u_{0}} \right \|_{W^{2,p_{0} }(\Omega )}\leq C_{15}^{'}\left \| u-\overline{u_{0}} \right \|_{L^{p_{0}}(\Omega )}\leq C_{15}^{''}e^{(-w)^{1/\beta}t},$$
 with $w=\frac{M}{p_{0}}.$ This completes the proof.
 \end{proof}
\begin{remark}
    With the aid of  fractional differential inequality, it is proved that the exponential stabilization toward the constant steady states by using the Lyapunov functional again. The most important thing is that we use Hardy-Little-wood-Sobolev inequality instead of Poincaré's inequality, and we get the above conclusion.
\end{remark}  
\begin{lemma}\label{89jhgtt}
Assume $\alpha\in(1,2),\beta\in(0,1)$ and based on using fractional Duhamel type integral equation  to equation \eqref{4.17}, we can get 
   \begin{align*}
  \nonumber  w(t)&=E_{\beta }(-t^{\beta }A)w(\tau_{0})\\
&\quad-\int_{\tau_{0}}^{t}(t-s)^{\beta -1}E_{\beta ,\beta }(-(t-s)^{\beta }A)(-\Delta )^{\alpha /2}\left ( (\rho (v(s))-\rho_{0})u(s) \right )ds
\end{align*}
where $A=(-\Delta)^{\alpha/2}$ and $w(t)=u-\overline{u_{0}}$.
 Then there exist constants $C=C(n,\Omega,\rho ,u_{0})>0$ and $\theta\in (0,1)$ such that
    $$\left \| u \right \|_{C^{2+\theta,1+\frac{\theta}{2} }(\bar{\Omega } \times [t,t+1])}\leq C,\qquad \forall t\geq1.$$
\end{lemma}
\begin{proof}
 Due  $\left \| u \right \|_{L^{\infty }(\Omega )}$ and $\left \| v \right \|_{W^{1,\infty }(\Omega )}$ are  uniform-in-time bounded, it can be inferred from \eqref{2.1} that
    \begin{align*}
        \underset{t>0}{\text{sup}}\left \| _{0}^{C}\textrm{D}_{t}^{\beta }v \right \|_{L^{\infty }(\Omega )}&\leq \underset{t>0}{\text{sup}}\left ( \left \| (I+(-\Delta )^{\alpha /2})^{-1}\left [ u\rho  (v) \right ] \right \|_{L^{\infty }(\Omega )}+\left \| u\rho  (v) \right \|_{L^{\infty }(\Omega )} \right )\\
&\leq C_{16}\,\underset{t>0}{\text{sup}}\left \| u\rho  (v) \right \|_{L^{\infty }(\Omega )}<\infty
    \end{align*}
and from the second equation of \eqref{1.1}
$$ \underset{t>0}{\text{sup}}\left \| v \right \|_{W^{2,p}(\Omega )}\leq C_{17} \,\underset{t>0}{\text{sup}}\left \| u \right \|_{L^{\infty }(\Omega )}< \infty ,$$
with any $1<p<\infty$ and $C_{16},C_{17}>0$ depending on $n,p$ and $\Omega$. Then, we get $v\in W_{p}^{2,1}(\Omega \times [t,t+1])$ with any $p>\frac{n+2}{2}$ for any $t>0$ and on the same time, by  Sobolev embedding theorem in \cite{ahn2019global}, there exist $\theta_{1}\in (0,2-\frac{n+2}{p}]$ such that
\begin{equation}
    \left \| v \right \|_{C^{\theta_{1},\frac{\theta _{1}}{2} }(\bar{\Omega } \times [t,t+1])}\leq C_{17}\left \| v \right \|_{W_{p}^{2,1}(\Omega \times [t,t+1])}\leq C_{17}^{'},\quad \forall t>0.
\end{equation}
which constant $C_{17},C_{17}^{'}>0$ independent of time. Similarly, referring to the method in [\cite{ahn2019global},Lemma 5.1], we obtain that
$$\left \| u \right \|_{C^{\theta_{2},\frac{\theta _{2}}{2} }(\bar{\Omega } \times [t,t+1])}\leq C_{18},\qquad \forall t\geq1,$$
where time-independent constant $C_{18}>0$.

Next, We combine the first equation and the second equation of \eqref{1.1} to obtain the following key identity: 
$$_{0}^{C}\textrm{D}_{t}^{\beta }v+\rho  (v)(-\Delta )^{\alpha /2}v=\left ( I+(-\Delta )^{\alpha /2} \right )^{-1}[\rho  (v)u]-\rho  (v)v.$$
Since $\rho(v(x,t))$ is uniformly bounded from above and below. \footnote{Referring to uniform elliptic operator $\rho(v(x,t))\Delta$ method in [\cite{jiang2022boundedness},Lemma 4.6], then we obtain $v\rho(v)$ and $\left ( I+(-\Delta )^{\alpha /2} \right )^{-1}[\rho (v)u]$ are now bounded in $C^{\theta_{1},\frac{\theta _{1}}{2} }(\bar{\Omega } \times [t,t+1]).$ in view of our assumption (H0). Further, through Schauder's theory, we get the inequality \eqref{huj897h} and \eqref{ju80okhyt}.}Thus we can further deduce, by a standard version of Schauder's theory for parabolic equations, that with some $\theta_{3}\in(0,1),$
\begin{equation}\label{huj897h}
    \left \| v \right \|_{C^{2+\theta_{3},1+\frac{\theta _{3}}{2} }(\bar{\Omega } \times [t,t+1])}\leq C_{18}^{'},\qquad \forall t\geq1.
\end{equation}
Conversely, we can also finally deduce the following inequalities from the equation of $u$ according to Schauder's theory
\begin{equation}\label{ju80okhyt}
    \left \| u \right \|_{C^{2+\theta,1+\frac{\theta}{2} }(\bar{\Omega } \times [t,t+1])}\leq C_{18}^{''},\qquad \forall t\geq1.
\end{equation}
Through the above preparations and fractional Duhamel type integral equation, we can now  prove the exponential decay of $\left \| u-\overline{u_{0}} \right \|_{L^{\infty}}$. Denotinng $w=u-\overline{u_{0}}$ and $\rho_{0}=\rho(\overline{u_{0}}),$ by the [\cite{li2018cauchy},(2.21)] and let $A=(-\Delta )^{\alpha /2}$, we infer from \eqref{4.17} that, for any $t>\tau_{0}\geq 1$,
\begin{align}
  \nonumber  w(t)&=E_{\beta }(-t^{\beta }A)w(\tau_{0})\\
&\quad-\int_{\tau_{0}}^{t}(t-s)^{\beta -1}E_{\beta ,\beta }(-(t-s)^{\beta }A)(-\Delta )^{\alpha /2}\left ( (\rho  (v(s))-\rho _{0})u(s) \right )ds
\end{align}
As a result, through Lemma \ref{lemma2.4}, we infer that
\begin{align*}
    &\left \|  w(t)\right \|_{L^{\infty }(\Omega )}\leq \left \|  E_{\beta }(-t^{\beta }A)w(\tau_{0})\right \|_{L^{\infty }(\Omega )}\\
&\quad+\int_{\tau_{0}}^{t}(t-s)^{\beta -1}\left \|  E_{\beta ,\beta }(-(t-s)^{\beta }A)(-\Delta )^{\alpha /2}\left ( (\rho  (v(s))-\rho _{0})u(s) \right )\right \|_{L^{\infty }(\Omega )}ds\\
&\leq \left \| w(\tau_{0}) \right \|_{L^{\infty }(\Omega )}+\frac{1}{\Gamma (\beta )}\int_{\tau_{0}}^{t}(t-s)^{\beta -1}\left \| (-\Delta )^{\alpha /2}\left ( (\rho  (v(s))-\rho _{0})u(s) \right ) \right \|_{L^{\infty }(\Omega )}ds.
\end{align*}
 Since $\left \| D^{\alpha }u \right \|_{L^{\infty }(\Omega )}\leq C_{19}$ with some $C_{19}>0$ for all $t\geq 1$. due to Lemma \ref{lemmader} and \eqref{556678}, we obtain that
\begin{align*}
    \left \| D^{\alpha}\left ( (\rho  (v(t))-\rho _{0})u(t) \right ) \right \|_{L^{\infty }(\Omega )}&\leq \left \| {\rho  }'(v(t))u(t)D^{\alpha }v(t) \right \|_{L^{\infty }(\Omega )}\\
&\quad+\left \| (\rho  (v(t))-\rho  (\overline{u_{0}}))D^{\alpha }u(t)) \right \|_{L^{\infty }(\Omega )}\\
&\leq C_{20}\left \| D^{\alpha }v(t) \right \|_{L^{\infty }(\Omega )}+C_{21}\left \|\rho  (v(t))-\rho  (\overline{u_{0}})  \right \|_{L^{\infty }(\Omega )}\\
&\leq C_{21}^{'}\left ( \left \| D^{\alpha }v(t) \right \|_{L^{\infty }(\Omega )}+\left \| v(t)- \overline{u_{0}}\right \|_{L^{\infty }(\Omega )} \right ),
\end{align*}
with all $t\geq 1$. And in order to prove above inequality, we also use the following formula that 
\begin{equation*}
    \left | \rho  (v)-\rho  (\overline{u_{0}}) \right |=\left | (v(t)-\overline{u_{0}}) \right |\int_{0}^{1}{\rho  }'(sv+(1-s)\overline{u_{0}})ds\leq C_{22}\left | v(t)- \overline{u_{0}} \right |,
\end{equation*}
since $sv+(1-s)\overline{u_{0}}$ is uniformly bounded from above and below on $[0,+\infty)\times \bar{\Omega }$ for all $s\in [0,1].$

Finally,  recalling Lemma \ref{lemma4.5} and [\cite{jiang2022boundedness},(4.25)], we may infer that
\begin{align}\label{4.25}
  \nonumber  \left \|  w(t)\right \|_{L^{\infty }(\Omega )}&\leq \left \| w(\tau_{0}) \right \|_{L^{\infty }(\Omega )}+\frac{1}{\Gamma (\beta )}\int_{\tau_{0}}^{t}(t-s)^{\beta -1}e^{(-w)^{1/\beta}s}ds\\
    &\leq C_{23}e^{(-w^{'})^{1/\beta}t}.
\end{align}
with any $w^{'}<w$ and $C_{23}>0$.
\end{proof}

\textbf{Proof of Theorem \ref{theorem4}:}
  Boundedness: Through the Lemma \ref{lemma4.3}, with the proof method in [\cite{ahn2019global},Lemma 4.3]    we can deduce the uniform-in-time boundedness of solutions. Further assuming that $\rho(v)=v^{-k}$ and $n\geq 4$, then referring to \cite{jiang2022boundedness}, we can prove that equation \eqref{1.1} has a unique global bounded classical solution, assuming $k\leq 1$ when $n=4,5$,or $k<\frac{4}{n-2}$ when $n\geq 6$.

Next, in order to obtain inequality \eqref{o90875},  by Lemma \ref{lemma4.5}-\ref{89jhgtt}, we have
    \begin{equation}\label{890oiu}
    \left \| v-\overline{u_{0}} \right \|_{W^{1,\infty }(\Omega )}\leq Ce^{(-w)^{1/\beta}t},\qquad \forall t>0,0<\beta<1.
\end{equation}
with $w=\frac{M}{p_{0}}$. Then from \eqref{4.25}, we get
\begin{equation}\label{jnhfcdse}
    \left \| u(\cdot ,t)-\overline{u_{0}} \right \|_{L^{\infty }(\Omega )}\leq C_{23}e^{(-w^{'})^{1/\beta}t}.
\end{equation}
with any $w^{'}<w$ and $C_{23}>0$. Finally, combined \eqref{890oiu} and \eqref{jnhfcdse}, we conclude that
    $$\left \| u(\cdot ,t)-\overline{u_{0}} \right \|_{L^{\infty }(\Omega )}+\left \| v(\cdot ,t)-\overline{u_{0}} \right \|_{W^{1,\infty }(\Omega )}\leq C_{24}e^{(-w^{'})^{1/\beta}t},\quad \forall t\geq 1,$$
with some $w^{'}>0,C_{24}>0$ depending $u_{0},\gamma,n$ and $\Omega$.
 \begin{remark}
     By using fractional Duhamel type integral equation, we can get the expression of the solution of the equation \eqref{4.17}:
   \begin{align*}
  \nonumber  w(t)&=E_{\beta }(-t^{\beta }A)w(\tau_{0})\\
&\quad-\int_{\tau_{0}}^{t}(t-s)^{\beta -1}E_{\beta ,\beta }(-(t-s)^{\beta }A)(-\Delta )^{\alpha /2}\left ( (\gamma (v(s))-\gamma_{0})u(s) \right )ds
\end{align*}
where $A=(-\Delta)^{\alpha/2}$ and $w(t)=u-\overline{u_{0}}$. Then, recalling Lemma \ref{lemma2.4} and Lemma \ref{lemmader}, we can derive the exponential stabilization of the global solutions by combining Lemma \ref{lemma4.5} and \eqref{4.25}.
 \end{remark}

\begin{appendices}

\section{ Definitions, complements and computations}\label{90oiuynytt}
\subsection{Definition of the fractional derivative }

\begin{definition}\textsuperscript{\cite{Kilbas2006}}
 Assume that $X$ is a Banach space and let $u:\left [ 0,T \right ]\rightarrow X$. The Riemann-Lioville fractional derivative operators of $u$ is defined by
\begin{eqnarray*}
_{0}\textrm{D}_{t}^{\beta }u(t)&=&\frac{1}{\Gamma (1-\beta )}\frac{d}{dt}\int_{0}^{t}(t-s)^{-\beta }u(s)ds,\\
\nonumber_{t}\textrm{D}_{T}^{\beta }u(t)&=&\frac{-1}{\Gamma (1-\beta )}\frac{d}{dt}\int_{t}^{T}(s-t)^{-\beta }u(s)ds,
\end{eqnarray*}
 where $\Gamma (1-\beta )$ is the Gamma function. The above integrals are called the left-sided and the right-sided the Riemann-Lioville  fractional derivatives.
\end{definition}
\begin{definition}\textsuperscript{\cite{li2018some}}
    Let $0<\beta<1$. Consider $u\in L_{loc}^{1}((0,T);\mathbb{R})$ such that $u$ has a right limit $u(0+)$ at $t=0$ in the sense of Definition \ref{definition1}. The $\beta-$th order  Caputo derivative of $u$ is a distribution in $\mathscr{D}'(\mathbb{R})$ with support in $(0,T]$, is defined by
    $$\partial _{c}^{\beta }u:=J_{-\beta }u-u_{0}\text{g}_{1-\beta }=\text{g}_{-\beta }*(\theta (t)(u-u_{0})),$$
    where $J_{\beta}$ denotes the fractional integral operator
    \begin{equation}\label{289jygbd}
        J_{\beta}u(t)=\frac{1}{\Gamma (\beta )}\int_{0}^{t}(t-s)^{\beta -1}u(s)ds.
    \end{equation}
    Similarly, the $\beta-$th order right Caputo derivative of $u$ is a distribution in $\mathscr{D}'(\mathbb{R})$ with support in $(-\infty,T]$, given by
    $$\tilde{\partial}_{T}^{\beta}u:=\tilde{\text{g}}_{-\beta}*(\theta(T-t)(u(t)-u(T-))). $$
\end{definition}

\begin{definition}\textsuperscript{\cite{li2018some}}\label{definition1}
    Let $B$ be a Banach space. For a function $u\in L_{loc}^{1}((0,T);B)$, if there exists $u_{0}\in B$ such that
    $$\lim_{t\rightarrow 0^{+}}\frac{1}{t}\int_{0}^{t}\left \| u(s)-u_{0} \right \|_{B}ds=0.$$
    We call $u_{0}$ the right limit of $u$ at $t=0$, denoted by $u(0+)=u_{0}.$
    Similarly, we define $u(T-)$ to be the constant $u_{T}\in B$ such that
    $$\lim_{t\rightarrow T^{-}}\frac{1}{T-t}\int_{t}^{T}\left \| u(s)-u_{T} \right \|_{B}ds=0.$$
\end{definition}

\begin{remark}

  As in \cite{li2018some}, we use the following distributions $\left \{ \text{g}_{\beta} \right \}$ as the convolution kernels for $\beta>-1:$
  \begin{equation*}
  \text{g}_{\beta}(t):= \begin{cases}
      \frac{\theta (t)}{\rho  (\beta )}t^{\beta -1},&\beta >0\\
\delta (t),&\beta =0\\
\frac{1}{\rho  (1+\beta )}D(\theta (t)t^{\beta }),&\beta \in(-1,0),
      \end{cases}
  \end{equation*}
Here $\theta (t)$ is the standard Heaviside step function and $D$ represents the distributional. $\text{g}_{\beta}$ can also be defined for $\beta\leq -1$(see \cite{li2018some}) so that these distributions form a convolution group $\Phi =\left \{ \text{g}_{\beta}:\beta \in\mathbb{R} \right \}$ and consequently we have
$$\text{g}_{\beta_{1}}*\text{g}_{\beta_{2}}=\text{g}_{\beta_{1}+\beta_{2}},\quad \forall \beta_{1},\beta_{2}\in \mathbb{R}.$$
Correspondingly, the time-reflected group:
$$\tilde{\Phi }:=\left \{ \tilde{\text{g}}_{\alpha}:\tilde{\text{g}}_{\alpha}(t)=\text{g}_{\alpha }(-t),\alpha \in\mathbb{R} \right \}.$$
Clearly, supp $\tilde{\text{g}}\subset (-\infty ,0]$ and for $\rho \in (0,1)$, the following equality is true
$$\tilde{\text{g}}_{-\rho  }(t)=-\frac{1}{\Gamma (1-\rho  )}D(\theta (-t)(-t)^{-\rho  })=-D\tilde{\text{g}}_{1-\rho  }(t),$$
where $D$ represents the distributional derivative on $t$.
\end{remark}

\subsection{Definition  of classical solution }\label{90oiuy}
\begin{definition}\textsuperscript{\cite{kemppainen2017representation}}\label{hhugtt}
  (Classical Solution)  Let $0<\beta\leq 1$ and $0<\alpha \leq 2$. Suppose $u_{0}\in C(\mathbb{R}^{n})$, Then  a function $u\in C\left ( (0,T ) \times\mathbb{R}^{n}\right )$ is a classical solution of the Cauchy problem
    \begin{equation}\label{uhffdrth}
        \begin{cases}
             \partial _{t}^{\beta }u+(-\Delta )^{\frac{\alpha}{2}}(\rho  (v)u)=0,& \text{in}\, \mathbb{R}^{n}\times (0,\infty)\\
             u(x,0)=u_{0}(x),&\text{in}\,\mathbb{R}^{n}
        \end{cases}
    \end{equation}
 \textbf{(i)}  $\mathcal{F}^{-1}(\left | \xi  \right |^{\alpha}\hat{\rho (v)u}(\xi ))(x)$ defines a continuous function of $x$ for each $t>0$,\\
 \textbf{(ii)} for every $x\in \mathbb{R}^{n}$, the fractional integral $J_{1-\beta}u$, as defined in \eqref{289jygbd}, is  continuously differentiable with respect to $t>0$,\\
\textbf{(iii)} the function $u(x,t)$ satisfies the integro-partial equation of \eqref{uhffdrth} for every $(x,t)\in \mathbb{R}^{n}\times (0,\infty)$ and the initial condition of \eqref{uhffdrth} for every $x\in \mathbb{R}^{n}$.
\end{definition}
\begin{theorem}\textsuperscript{\cite{jiang2022boundedness}}
    Let $\Omega$ be a smooth bounded domain of $\mathbb{R}^{n}$. Suppose that $\rho  (\cdot )$ satisfies (H0) and $u_{0}$ satisfies \eqref{1.5},$u_{0}\in C(\Omega),u_{0}\geq 0$.  Then there exists $T_{max}\in (0,\infty ]$ suth that problem \eqref{1.1} possesses a unique non-negative classical solution
    $(u,v)\in \left ( C\left ( \Omega\times (0,T_{max} ) \right )\cap C^{2,1}(\Omega \times (0,T_{max} )) \right )^{2}$ in the form of Definition \ref{hhugtt}.
Moreover,If $T_{max}<\infty$, then
    $$\lim_{t\rightarrow T_{max}}\text{sup}\,\left \| u(\cdot ,t) \right \|_{L^{\infty }(\Omega )}=\infty.$$
\end{theorem}
\begin{remark}
    The global existence and uniqueness of classical solutions to equation \eqref{1.1} is obtained by referring to [\cite{kemppainen2017representation},Definition 2.4] and [\cite{jiang2022boundedness},Theorem 2.1]. But the difference is that because of the existence of time-space fractional derivatives, the conditions are more than [\cite{jiang2022boundedness},Theorem 2.1].
\end{remark}
\subsection{ Definition  of fractional Laplacian }\label{8i90oyt}
In this section, we know some properties about fractional Laplacian.  By \cite{de2012general}, the nonlocal operator $(-\Delta )^{\frac{\alpha }{2}}$, known as the Laplacian of order $\alpha$, is defined for any function $g$ in the Schwartz class through the Fourier transform: if $(-\Delta )^{\frac{\alpha }{2}}g=h$, then
\begin{equation}\label{1.9087}
    \hat{h}(\xi )=\left | \xi  \right |^{\alpha }\hat{\text{g}}(\xi ).
\end{equation}
In addition, it also needs the following property. Thought \cite{de2012general}, if $\psi $ and $\varphi $ belong to the Schwartz class, definition \eqref{1.9087} of the fractional Laplacian together with Plancherel's theorem yields
\begin{align}\label{gyuij}
 \nonumber   \int _{\mathbb{R}^{n}}(-\Delta )^{\alpha /2}\psi \varphi dx&=\int _{\mathbb{R}^{n}}\left | \xi  \right |^{\alpha }\hat{\psi }\hat{\varphi }dx\\
\nonumber&=\int _{\mathbb{R}^{n}}\left | \xi  \right |^{\alpha/2 }\hat{\psi }\left | \xi  \right |^{\alpha/2 }\hat{\varphi }dx\\
&=\int _{\mathbb{R}^{n}}(-\Delta )^{\alpha /4}\psi (-\Delta )^{\alpha /4}\varphi dx
\end{align}
According to Chapter V in \cite{stein1970singular}, the nonlocal operator $(-\Delta )^{\frac{\alpha}{2}}$, known as the Laplacian of order $\frac{\alpha}{2}$, is given by the Fourier multiplier
$$D^{\alpha}u(x):=(-\Delta )^{\alpha/2}u(x):=\mathcal{F}^{-1}(\left | \xi  \right |^{\alpha}\hat{u}(\xi ))(x),$$
where $\hat{u}(\xi)=\mathcal{F}(u(x))$ is the Fourier transformation of function $u(x)$. Also, we will use the following formula as the one give in Caffarelli and Silvestre \cite{caffarelli2007extension}:
\begin{align}\label{4}
    (-\Delta)^{\alpha/2} u=C_{n,\alpha}P.V.\int _{\mathbb{R}^{n}}\frac{u(x)-u(y)}{\left | x-y \right |^{n+\alpha}}dy,
\end{align}	
where $C_{n,\alpha}=\frac{2^{\alpha-1}\alpha \Gamma ((n+\alpha)/2)}{\Gamma (1-\alpha/2 )\pi^{\alpha/n}}$ is normalization constant and $P.V.$ denotes the Cauchy principal value.
\subsection{Some useful Lemma  about  fractional derivative }\label{inequaliyy}

\begin{lemma}\textsuperscript{\cite{alikhanov2010priori}}\label{lemma23}
Let $0<\beta<1$ and $v\in C([0,T],\mathbb{R}^{N}),{v}'\in L^{1}(0,T;\mathbb{R}^{N})$ and $v$ be monotone. Then
  \begin{equation}\label{888}
    v(t)\partial _{t}^{\beta }v(t)\geq \frac{1}{2}\partial _{t}^{\beta }v^{2}(t),\quad t\in(0,T].
  \end{equation}
\end{lemma}

\begin{lemma}\textsuperscript{\cite{Kilbas2006}}\label{lemma200}
    Assume $0<\beta<2$, for any $\gamma \in \mathbb{R}$, there is a constant $\mu$ such that $\frac{\pi \beta }{2}<\mu <min\left \{ \pi ,\pi \beta\right \}$, then there is a constant $c=c(\beta ,\gamma ,\mu )>0$ , such that$$\left | E_{\beta ,\gamma }(z)\right |\leq \frac{c}{1+\left | z \right |}, \qquad \mu \leq \left | arg(z) \right |\leq \pi.$$
\end{lemma}

 \begin{lemma}\textsuperscript{\cite{2001Fractional}}\label{Gao}
  Let us consider the  fractional differential equation
  \begin{align}\label{jiojio}
  \begin{cases}
     _{0}^{C}\textrm{D}_{t}^{\beta }u(t)=wu(t),\,0<\beta<1,w>0,\\
u(0)=u_{0}.
  \end{cases}
 \end{align}
Then, the solution of \eqref{jiojio} can be obtained by applying the Laplace transform technique which implies:
\begin{equation}\label{9i9i}
    u(t)=u_{0}E_{\beta}(wt^{\beta}),\quad t>0.
\end{equation}
\end{lemma}
\begin{lemma}\textsuperscript{\cite{2001Fractional}}\label{lemma390}
   If $0<\beta<1,t>0,w>0$, for Mittag-Leffler function $E_{\beta,1 }(wt^{\beta })$ , then there is a constant $C$ such that
   \begin{equation}\label{frfr}
     E_{\beta}(wt^{\beta})=E_{\beta,1 }(wt^{\beta })\leq Ce^{w^{\frac{1}{\beta }} t}.
   \end{equation}
\end{lemma}

\begin{lemma}\textsuperscript{\cite{2016Weak}}\label{lemma2.14}
 Suppose that a nonnegative function $u(t)\geq 0$  satisfies
 \begin{equation}\label{4433}
   _{0}^{C}\textrm{D}_{t}^{\beta }u(t)+c_{1}u(t)\leq f(t)
 \end{equation}
  for almost all $t\in [0,T]$, where $c_{1}>0$, and the function $f(t)$ is nonnegative and integrable for $t\in [0,T]$. Then
  \begin{equation}\label{55}
    u(t)\leq u(0)+\frac{1}{\Gamma (\beta)}\int_{0}^{t}(t-s)^{\beta-1}f(s)ds.
  \end{equation}
 \end{lemma}
 \begin{lemma}\textsuperscript{\cite{Ahmed2017A}}\label{lemma1}
Let $0<\beta<1$ and $u\in C([0,T],\mathbb{R}^{N}),{u}'\in L^{1}(0,T;\mathbb{R}^{N}),p\geq 2$ and $u$ be nonnegative and monotone. Then there is
\begin{equation*}
   u^{p-1}(_{0}^{C}\textrm{D}_{t}^{\beta }u) \geq \frac{1}{p}(_{0}^{C}\textrm{D}_{t}^{\beta }u^{p}).
\end{equation*}
\end{lemma}
\begin{lemma}\textsuperscript{\cite{Ahmed2017A}}\label{lemma0099}
    Let one of the following conditions be satisfied:\\
\textbf{(a)} $u\in C([0,T]),v\in C^{\gamma}([0,T]),\beta<\gamma\leq 1$;\\
\textbf{(b)} $v\in C([0,T]),u\in C^{\gamma}([0,T]),\beta<\gamma\leq 1$;\\
\textbf{(c)} $u\in C^{\gamma}([0,T]),v\in C^{\delta}([0,T]),\beta<\gamma+\delta,0<\gamma<1,0<\delta< 1$.\\
Then we have:
    \begin{align*}
        &_{0}\textrm{D}_{t}^{\beta }(uv)(t)=u(t)_{0}\textrm{D}_{t}^{\beta }v(t)+v(t)_{0}\textrm{D}_{t}^{\beta }u(t)\\
&\qquad-\frac{1}{\Gamma (1-\beta )}\int_{0}^{t}\frac{(u(s)-u(t))(v(s)-v(t))}{(t-s)^{\beta +1}}ds-\frac{u(t)v(t)}{\Gamma (1-\beta  )t^{\beta  }}
    \end{align*}
\end{lemma}
\begin{remark}
    Through the above conclusions and reference \cite{Ahmed2017A} , we get the following two immediate consequence are:\\
\textbf{(1)} If $u$ and $v$ have the same sign and are both increasing or both decreasing, then
        \begin{equation}\label{3.133}
            _{0}\textrm{D}_{t}^{\beta }(uv)(t)\leq u(t)_{0}\textrm{D}_{t}^{\beta }v(t)+v(t)_{0}\textrm{D}_{t}^{\beta }u(t).
        \end{equation}
        \textbf{(2)} For the Caputo derivative, inequality \eqref{3.133} reads
        \begin{align}
         \nonumber   &_{0}^{C}\textrm{D}_{t}^{\beta }(uv)(t)\leq u(t)_{0}^{C}\textrm{D}_{t}^{\beta }v(t)+v(t)_{0}^{C}\textrm{D}_{t}^{\beta }u(t)\\
           &\qquad +\frac{t^{-\beta }}{\Gamma (1-\beta )}(u(t)v(0)+v(t)u(0)-u(0)v(0)).
        \end{align}

\end{remark}

\subsection{ Some functional inequalities related to  fractional Laplacian }\label{i9io87h}
\begin{lemma}\textsuperscript{\cite{Ahmed2017A}}
    If $u,v\in C_{0}^{\infty}(\mathbb{R}^{n})$ and $0\leq \alpha \leq 2$, then
    \begin{align}\label{39087}
   \nonumber     &u(-\Delta )^{\frac{\alpha }{2}}v+v(-\Delta )^{\frac{\alpha }{2}}u-(-\Delta )^{\frac{\alpha }{2}}(uv)\\
        &=A_{\alpha /2}\int _{\mathbb{R}^{n}}\frac{(u(x)-u(y))(v(x)-v(y))}{\left | x-y \right |^{n+\alpha }}dy
    \end{align}
    where $A_{\alpha /2}>0$ and $\frac{A_{\alpha /2}}{\alpha /2(1-\alpha /2)}$ has finite,positive limits as $\alpha \rightarrow 0$ and $\alpha \rightarrow 2$.
\end{lemma}
\begin{remark}
    If $u$ and $v$ have the same sign and are both increasing or both decreasing, then the equation \eqref{39087} has the following inequality
    \begin{equation}\label{556678}
        (-\Delta )^{\frac{\alpha }{2}}(uv)\leq u(-\Delta )^{\frac{\alpha }{2}}v+v(-\Delta )^{\frac{\alpha }{2}}u.
    \end{equation}
\end{remark}

\begin{lemma}\textsuperscript{\cite{Ahmed2017A}}\label{lemmader}
  Let $\theta\in C_{0}^{2}(\mathbb{R}^{n})$ and $\Phi $ be a convex function of one variable. Then
  \begin{equation}\label{33.11}
     {\Phi }'(\theta )(-\Delta )^{\frac{\alpha }{2}}\theta (x)\geq (-\Delta )^{\frac{\alpha }{2}}\Phi (\theta )(x).
  \end{equation}
\end{lemma}
\begin{lemma}\textsuperscript{\cite{de2012general}}\label{lemmagn}
    (Nash-Gagliardo-Nirenberg-Type Inequality). Let $p\leq 1,r>1$, and $0<\alpha<2$. There is a constant $C=C(p,r,\alpha,n)>0$ such that for any $\xi\in L^{p}(\mathbb{R}^{n})$ with $(-\Delta )^{\alpha/4 }\xi\in L^{r}(\mathbb{R}^{n})$ we have
    \begin{equation}
        \left \| \xi \right \|_{r_{2}}^{\theta +1}\leq C\left \| (-\Delta )^{\alpha/4 }\xi  \right \|_{r}\left \| \xi \right \|_{p}^{\theta },\quad r_{2}=\frac{n(rp+r-p)}{r(n-\alpha /2)},\,\theta =\frac{p(r-1)}{r}.
    \end{equation}
\end{lemma}

\begin{lemma}\textsuperscript{\cite{hardy1928some}}\label{lemmaH}
    (Hardy-Little-wood-Sobolev Inequality) For every $v$ such that $(-\Delta )^{\alpha/4 }v\in L^{2}(\mathbb{R}^{n}),0<\alpha<2$, it hold that
    \begin{equation*}
        \left \| v \right \|_{r_{1}}\leq c(n,\alpha )\left \| (-\Delta )^{\alpha /4}v\right \|_{2},\quad r_{1}=\frac{2n}{n-\alpha }.
    \end{equation*}
\end{lemma}
\begin{lemma}\textsuperscript{\cite{bonforte2014quantitative}}\label{lemmawer}
    (Stroock-Varopoulos' inequality)\,Let $0<\frac{\alpha}{2}<1,p>1$, then
    \begin{equation}
        -\int _{\Omega }\left | f \right |^{p-2}fD^{\alpha }fdx\leq -\frac{4(p-1)}{p^{2}}\left \| D^{\frac{\alpha }{2}}f^{\frac{p }{2}} \right \|_{2}^{2}
    \end{equation}
    for all $f\in L^{p}(\Omega)$ such that $D^{\alpha}f\in L^{p}(\Omega).$
\end{lemma}

\begin{lemma}\textsuperscript{\cite{ taylorremarks}}\label{lemma2.4}\\
\textbf{(i)} Suppose that $e^{-tA}$ is a contraction semi-group in a Banach space, where $A$ is the generator of the semigroup. Then,
    \begin{equation*}
        \left \| E_{\beta }(-t^{\beta }A) f\right \|_{B}\leq \left \| f \right \|_{B}, \qquad   \left \| E_{\beta,\beta  }(-t^{\beta }A) f\right \|_{B}\leq \frac{1}{\Gamma (\beta )}\left \| f \right \|_{B}.
    \end{equation*}
\textbf{(ii)} Let $0<\alpha\leq 2$ and $A=(-\Delta )^{\alpha/2}$. If $1<p<\infty$ and $\sigma \in (0,1]$, then for $T_{0}>0,$ there exists $C>0$ such that
     \begin{equation*}
        \left \| E_{\beta }(-t^{\beta }A) f\right \|_{H^{\sigma \alpha ,p}}\leq Ct^{-\sigma \beta }\left \| f \right \|_{p}, \quad   \left \| E_{\beta,\beta  }(-t^{\beta }A) f\right \|_{H^{\sigma \alpha ,p}}\leq Ct^{-\sigma \beta }\left \| f \right \|_{p}.
    \end{equation*}
    uniformly for $t\in (0,T_{0}].$
\end{lemma}

\section{ Relevant complements to proof of Theorem \ref{theorem4} }\label{ijugggrc}
\subsection{Assumptions and properties about function $\rho(s)$}\label{367897}
Through reference \cite{jiang2022boundedness}, the function $ \rho(\cdot )$ is a given function satisfying the following conditions:
\begin{align}\label{1.5}
  \nonumber  &\text{(H0)}:\qquad  \rho (\cdot )\in C^{3}[0,+\infty ), \rho (\cdot )>0,{ \rho }'(\cdot )\leq 0\quad\text{on}\,(0,+\infty ),\lim_{s\rightarrow +\infty } \rho (s)=0.\\
  &\text{(H1)}:\qquad  \rho (s)+s{ \rho }'(s)\geq 0,\qquad \forall s>0,\\
  &\text{(H2)}:\qquad \text{there is}\,k>0\,\text{such that}\,\lim_{s\rightarrow +\infty }s^{k} \rho (s)=+\infty.\\
  &\text{(H3)}:\qquad l_{0}\left | { \rho }'(s) \right |^{2}\leq  \rho(s){ \rho }''(s),\,\text{with some}\,l_{0}> \frac{n+2}{4}\,\text{for all}\,s>0.
\end{align}
\begin{lemma}\textsuperscript{\cite{ahn2019global}}\label{lemma2.2}
    Suppose that $(u,v)$ is the classical solution of \eqref{1.1} up to the maximal time of existence $T_{max}\in (0,\infty ]$. Then, there exists a strictly positive constant $v_{*}=v_{*}(n,\Omega ,\left \| u_{0} \right \|_{L^{1}(\Omega )})$ such that, for all $t\in(0,T_{max})$, it holds that
    $$\underset{x\in \Omega }{\text{inf}}v(x,t)\geq v_{*}.$$
\end{lemma}

\begin{remark}
     By reference \cite{jiang2022boundedness}, in view of the time-independent lower bound $0<v_{*}\leq v(x,t),$ one can slightly weaken assumption (H1) as follows:
    \begin{equation}\label{3.4}
        \rho  (s)+s{\rho  }'(s)\geq 0,\quad \forall s\geq v_{*}.
    \end{equation}
    On the other hand,  a direct calculation indicates that the above assumption yields that
    \begin{equation}\label{3.5}
        s\rho  (s)\geq v_{*}\rho  (v_{*}),\quad \forall s\geq v_{*},
    \end{equation}
    and hence $\rho $ fulfills (H2) with any $k>1$.
\end{remark}

\begin{remark}
   By references \cite{jiang2022boundedness},  under above assumption H0-H3, it may infer that there exist $b>0$ and $s_{b}>v_{*}$ such that, for all $s\geq s_{b},$
    \begin{equation}\label{app3}
    \footnote{It is mainly obtained from the property of $\rho$ in reference \cite{jiang2022boundedness}.}1/\rho (s)\leq bs^{k},
\end{equation}
and on the one hand, since $\rho (\cdot )$ is non-increasing,
\begin{equation}\label{app4}
    \footnote{It is mainly obtained from the non-increasing of $\rho$ in reference \cite{jiang2022boundedness}.}1/\rho (s)\leq 1/\rho (s_{b})
\end{equation}
for all $0\leq s<s_{b}$. Therefore,for all $s\geq 0$, it holds that
\begin{equation}\label{3.7}
       \footnote{The following formula can be obtained by combining \eqref{app3} and \eqref{app4}.} 1/\rho (s)\leq bs^{k}+1/\rho (s_{b}).
    \end{equation}
On the other hand, thanks to \eqref{3.7}, one has 
\begin{align}\label{app5}
 \nonumber    \int _{\Omega }v^{p-1}\rho  (v)udx&\geq\int _{\Omega }v^{p-1}(bv^{k}+1/\rho (s_{b}))^{-1}udx\\
    &\geq C_{1}\int _{\Omega }(v_{k}+1)^{-1}v^{p-1}udx.
\end{align}
with $1/C_{1}=\frac{max \left \{  b\rho (s_{b}),1\right \} }{\rho (s_{b})}>0$ independent of $p$ and time, in the view of the fact that
\begin{equation}\label{app6}
    \footnote{The main reference \cite{jiang2022boundedness} is available. } bv^{k}+1/\rho (s_{b})=\frac{1}{\rho (s_{b})} (bv^{k}\rho (s_{b})+1)\leq \frac{max \left \{  b\rho (s_{b}),1\right \} }{\rho (s_{b})}(v^{k}+1).
\end{equation}
Since $v^{k}\geq v_{*}^{k}$ by Lemma \ref{lemma2.2},it holds that
\begin{equation}\label{78uuyh}
    \footnote{The main reference \cite{jiang2022boundedness} is available. } (v^{k}+1)^{-1}v^{p-1} \geq (v^{k}+v_{*}^{-k}v^{k})^{-1} v^{p-1}=\frac{v^{p-k-1} }{1+v_{*}^{-k} }
\end{equation}
from which we deduce that
\begin{equation}\label{8hfhjn}
     \footnote{The main reference \cite{jiang2022boundedness} is available. } \int _{\Omega }v^{p-1}\rho  (v)udx\geq C_{2}\int _{\Omega }v^{p-k-1}udx,
\end{equation}
where $C_{2}>0$ may depend on the initial datum, $n,\Omega$ and $\rho$, but is independent of $p$ and time.
\end{remark}

\subsection{ Optimization  proposition and lemma}\label{0poiuytgb}
We state and prove here some simple technical proposition and lemma that has been used in the proof of Theorem \ref{theorem4}.
\begin{lemma}\label{lemma3.4}
    Assume that $n\geq3$. Suppose that $\rho$ satisfies (H0) and (H2) with some $0<k<\frac{4}{n-2}.$ Let $L>1$ be a generic constant. There exists $C_{0}>0$ depending only on the initial datum, $\Omega,K,\lambda_{1},\lambda_{2}$ and $n$ such that, for any $p>q\geq q_{*}=\frac{2n}{n-2}$ satisfying
    $$q<p=2q-\frac{nk}{2},$$
    it holds that
    \begin{equation}\label{89hygfred}
        _{0}^{C}\textrm{D}_{t}^{\beta }\int _{\Omega }v^{p}dx+\lambda _{2}p\int _{\Omega }v^{p}dx\leq C_{0}L^{\frac{n}{2} }p^{\frac{n+2}{2} }\left ( \int _{\Omega }v^{p}dx \right ) ^{2}.
    \end{equation}
\end{lemma}
\begin{proof}
    The important proof of this Lemma mainly refers to [\cite{jiang2022boundedness},Lemma3.4]. In the reference \cite{jiang2022boundedness}, author uses Hölder's inequality , Sobolev embedding inequality and Young's inequality to shrink $ 2\lambda _{2}p\int _{\Omega }v^{p}dx $. Then we have that
    $$ 2\lambda _{2}p\int _{\Omega }v^{p}dx \leq \frac{\lambda _{1}p(p-k-1) }{L(p-k)^{2} } \left \| v^{\frac{p-k}{2} }  \right \|_{H^{1}(\Omega )}^{2}+C_{0}L^{\frac{n}{2} }p^{\frac{n+2}{2} }\left ( \int _{\Omega }v^{p}dx \right ) ^{2} .$$
    Combining Lemma \ref{lemma3.3} and recalling that $L>1$, we finally arrive at the following inequality:
    $$ _{0}^{C}\textrm{D}_{t}^{\beta }\int _{\Omega }v^{p}dx+\lambda _{2}p\int _{\Omega }v^{p}dx\leq C_{0}L^{\frac{n}{2} }p^{\frac{n+2}{2} }\left ( \int _{\Omega }v^{p}dx \right ) ^{2}.$$
    This completes the proof.
\end{proof}
\begin{remark}
    The Lemma mainly delates $2\lambda _{2}p\int _{\Omega }v^{p}dx$ through relevant inequalitys and [\cite{jiang2022boundedness},Lemma 3.4]. Further, combined with Lemma \ref{lemma3.3}, we will  obtain the  inequality \eqref{89hygfred}.
\end{remark}

\begin{proposition}\label{ijyuj90977}
     Assume that $n\geq3$ and $\rho$ satisfies (H0),(H1) and (H2) in Appendix \ref{367897} with some $k\in(0,\frac{4}{n-2})$. First, for $(p,q)=(p_{r},p_{r-1})$, we can get
    $$_{0}^{C}\textrm{D}_{t}^{\beta }\int _{\Omega }v^{p_{r}}dx+\lambda _{2}p_{r}\int _{\Omega }v^{p_{r}}dx\leq \lambda _{2}p_{r}\mathcal{B}_{r} (\mathcal{Q}_{r-1})^{2}.$$
 Both $\mathcal{B}_{r} $ and $\mathcal{Q}_{r}$ here are defined below.

Then there is $v^{*}>0$ depending only on the initial datum, $\rho,n$ and $\Omega$ such that
    \begin{equation}
        \underset{0\leq t< T_{max}}{sup} \left \| v(\cdot ,t) \right \|_{L^{\infty  } (\Omega )} \leq v^{*} .
    \end{equation}
\end{proposition}
\begin{proof}
\footnote{The definition of $p_{r},\mathcal{B}_{r},\mathcal{Q}_{r}$ is mainly referred to [\cite{jiang2022boundedness},Proposition 3.5], but the difference is that the inequality we get is fractional differential inequality rather than differential inequality.} For all $r\in \mathbb{N}$ we define  
    $$p_{r}\triangleq 2^{r}(q_{*}-\frac{nk}{2})+\frac{nk}{2},\qquad p_{0}=q_{*}.$$
Then $p_{r}>q_{*}>\frac{nk}{2}$ and $p_{r}=2p_{r-1}-\frac{nk}{2}.$ We apply Lemma \ref{lemma3.4} with $(p,q)=(p_{r},p_{r-1})$ to get
    $$_{0}^{C}\textrm{D}_{t}^{\beta }\int _{\Omega }v^{p_{r}}dx+\lambda _{2}p_{r}\int _{\Omega }v^{p_{r}}dx\leq \lambda _{2}p_{r}\mathcal{B}_{r} (\mathcal{Q}_{r-1})^{2},$$
where
    $$\mathcal{Q}_{r} \triangleq \underset{0\leq t< T_{max}}{sup}\int _{\Omega }v^{p_{r}}dx\quad \text{and}\quad \mathcal{B}_{r}\triangleq \frac{C_{0}L^{\frac{n}{2} }p^{\frac{n}{2} }}{\lambda _{2}}. $$
 Note that $\mathcal{Q}_{r}$ is finite for all $r$ in view of \eqref{2.2}. Now, letting $y(t)=\int _{\Omega }v^{p_{r}}dx$, we get
    $$_{0}^{C}\textrm{D}_{t}^{\beta }y(t)+\lambda _{2}p_{r}y(t)\leq  \lambda _{2}p_{r}\mathcal{B}_{r} (\mathcal{Q}_{r-1})^{2}.$$
   By Lemma \ref{lemma2.14}, we infer from the above that
    \begin{align}
        y(t)&\leq y(0)+\frac{1}{\Gamma (\beta) } \int_{0}^{t}(t-s)^{\beta-1}(\lambda _{2}p_{r}\mathcal{B}_{r} (\mathcal{Q}_{r-1})^{2})ds\\
        &\leq \left \| v_{0} \right \| _{L^{\infty }(\Omega ) }^{p_{r}} +\lambda _{2}p_{r}\mathcal{B}_{r} (\mathcal{Q}_{r-1})^{2}\int_{0}^{t}(t-s)^{\beta-1}ds\\
        &\leq \left \| v_{0} \right \| _{L^{\infty }(\Omega ) }^{p_{r}}+\frac{\lambda _{2}p_{r}\mathcal{B}_{r} (\mathcal{Q}_{r-1})^{2}T^{\beta } }{\beta \Gamma (\beta)}.
    \end{align}
As a result, we obtain that, for all $r\in \mathbb{N}$,
$$\mathcal{Q}_{r}=\underset{0\leq t< T_{max}}{sup}\int _{\Omega }v^{p_{r}}dx\leq max \left \{ \left \| v_{0} \right \| _{L^{\infty }(\Omega ) }^{p_{r}},\frac{\lambda _{2}p_{r}\mathcal{B}_{r} (\mathcal{Q}_{r-1})^{2}T^{\beta } }{\beta \Gamma (\beta)} \right \} .$$
Since $p_{r}\geq q_{*}$ for all $r\leq 1$, one can choose $L>1$ sufficiently large depending only on the initial datum, $\Omega,n$ and $k$ such that $\mathcal{B}_{r}>1$ for all $r\geq 1$. Moreover, adjusting $C_{0}$ by a proper larger nimber, we have that
$$\mathcal{B}_{r}\leq C_{0}a^{r},$$
with some $a>0$. Furthermore, since  Lemma \ref{lemma3.1}, Sobolev embedding $H^{1}\hookrightarrow L^{q_{*}}$ and $\rho$ satisfies (H0) and (H1),  we may find some large constant $K_{0}>1$ has always dominated $\left \| v_{0} \right \| _{L^{\infty}(\Omega)}$ and $\int _{\Omega }v^{q_{*}}dx$ for all time.

Iteratively, \footnote{By reference [\cite{alikakos1979application}, We do Alikakos-Moser iterate on $\int _{\Omega }v^{p_{r}}dx$ and combine fractional differential inequality to proof \eqref{ikjgtrf}.}we deduce that               
\begin{align*}
    \int _{\Omega }v^{p_{r}}dx&\leq max\left \{ \mathcal{B}_{r}\mathcal{B}_{r-1}^{2}\mathcal{Q}_{r-2}^{4}, \mathcal{B}_{r}K_{0}^{2p_{r-1}},K_{0}^{p_{r}}  \right \} \\
    &= max\left \{ \mathcal{B}_{r}\mathcal{B}_{r-1}^{2}\mathcal{Q}_{r-2}^{4}, \mathcal{B}_{r}K_{0}^{2p_{r-1}} \right \}\\
    &\leq \cdots \\
    &\leq max\left \{ \mathcal{B}_{r}\mathcal{B}_{r-1}^{2}\mathcal{B}_{r-2}^{4}\cdots\mathcal{B}_{1}^{2^{r-1}}\mathcal{Q}_{0}^{2^{r}}, \mathcal{B}_{r}\mathcal{B}_{r-1}^{2}\cdots\mathcal{B}_{2}^{2^{r-2}} K_{0}^{2^{r-1}p_{1}} \right \}\\
    &\leq max\left \{ \mathcal{B}_{r}\mathcal{B}_{r-1}^{2}\mathcal{B}_{r-2}^{4}\cdots \mathcal{B}_{1}^{2^{r-1}}K_{0}^{2^{r}}, \mathcal{B}_{r}\mathcal{B}_{r-1}^{2}\cdots\mathcal{B}_{2}^{2^{r-2}} K_{0}^{2^{r-1}p_{1}} \right \}\\
    &\leq C_{0}^{2^{0} +2^{1} +\cdots+2^{r-1} }\times a^{1\cdot r+2(r-1)+2^{2}(r-2)+\cdots +2^{r-1} (r-(r-1)) }\times \tilde{K} _{0}^{2^{r} } \\
    &=C_{0}^{2^{r}-1}a^{2^{1+r}-r-2}\tilde{K} _{0}^{2^{r} },
\end{align*}
where $\tilde{K}=max\left \{ K_{0},K_{0}^{p-1} \right \} $. Finally, recalling that $p_{r}=2^{r}(q_{*}-\frac{nk}{2})+\frac{nk}{2}$, we deduce that
\begin{equation}\label{ikjgtrf}
    \left \| v \right \| _{L^{\infty }(\Omega ) } \leq \lim_{r \nearrow  \infty }\left ( C_{0}^{2^{r}-1}a^{2^{1+r}-r-2}\tilde{K} _{0}^{2^{r} }  \right )^{1/p_{r}}=\left ( C_{0}a^{2}\tilde{K} _{0} \right )^{\frac{2}{2q_{*}-nk} },
\end{equation}
The completes the proof.
\end{proof}
\begin{remark}
    The method of proof the Lemma is mainly referenced [\cite{jiang2022boundedness},Proposition 3.5], but in the process of proven, we used fractional differential inequality to narrow $y(t)=\int _{\Omega }v^{p_{r}}dx$ instead of differential equation.
\end{remark}
\begin{remark}
    Due to the establishment of Proposition \ref{ijyuj90977}, then $v$ has a uniform-in-time upper bound in $\bar{\Omega }\times [0,T_{max} )$ in the sense of fractional derivative for $\rho (\cdot )$ satisfies (H0),(H1) and (H3) in Appendix \ref{367897}.
\end{remark}

\begin{lemma}\textsuperscript{\cite{jiang2022boundedness}}\label{lemmaiojd}
    Assume that $\rho (\cdot)$ satisfies (H0),(H1) and (H3). For any $1+p\in(0,l_{0}^{2})$, there exist time-independent constants $q=\frac{pl_{0}}{2}>0$ and $\delta _{0}=\delta _{0}(p,q)\in (0,1)$ such that
    \begin{align*}
        &\frac{(p+1)(p+2q)^{2}}{4p(1-\delta_{0})}\int _{\Omega }u^{1+p}\rho  ^{q-1}\left | {\rho  }' \right |^{2}\left | \nabla v \right |^{2}dx\\
&\leq q\int _{\Omega }\left ( (p+q+1)\left | {\rho  }'(v) \right |^{2}+\rho  {\rho  }'' \right )u^{1+p}\rho  ^{q-1}\left | \nabla v \right |^{2}dx
    \end{align*}
\end{lemma}
\begin{remark}
    Through the above conclusions and references [\cite{jiang2022boundedness},Lemma 4.2], we further get the following inequality
    \begin{align}\label{lemma4.2}
    \nonumber   & \frac{(p+1)(p+2q)^{2}}{4p(1-\delta_{0})}\int _{\Omega }u^{1+p}\rho  ^{q-1}\left | {\rho  }' \right |^{2}\left | D^{\alpha/2 }v \right |^{2}dx\\
&\leq q\int _{\Omega }\left ( (p+q+1)\left | {\rho  }'(v) \right |^{2}+\rho  {\rho  }'' \right )u^{1+p}\rho  ^{q-1}\left | D^{\alpha/2 }v \right |^{2}dx
    \end{align}
\end{remark}

\end{appendices}


\bibliographystyle{elsarticle-num}
\bibliography{Ref}
\end{document}